\documentclass[reqno]{amsart}
\usepackage{hyperref}
\usepackage{amsmath}
\usepackage[safeinputenc=true,backref=true,style=alphabetic,firstinits,useprefix=true,url=false,maxnames=6]{biblatex}

\usepackage{mathrsfs}
\usepackage{amsthm}
\usepackage{amssymb}

\numberwithin{equation}{section}
\newtheorem{theorem}[equation]{Theorem}
\newtheorem*{claim}{\textit{Claim}}
\newtheorem{proposition}[equation]{Proposition}
\newtheorem{corollary}[equation]{Corollary}
\newtheorem{lemma}[equation]{Lemma}

\theoremstyle{definition}
\newtheorem{example}[equation]{Example}
\newtheorem{definition}[equation]{Definition}

\newcommand{\define}[1]{\textit{#1}}
\newcommand{\erdos}{Erd\H{o}s}
\newcommand{\szemeredi}{Szemer\'{e}di}
\newcommand{\folner}{F\o{}lner}

\newcommand{\cstar}{C$^*$}
\newcommand{\lcsc}{locally compact, second countable}
\newcommand{\lcsca}{locally compact, second countable, amenable}
\newcommand{\lcscWMa}{locally compact, second countable, WM, amenable}
\newcommand{\lmult}{l}
\newcommand{\rmult}{r}
\newcommand{\modular}{\triangle}
\renewcommand{\epsilon}{\varepsilon}
\newcommand{\intd}{\,\mathrm{d}}
\newcommand{\haar}{\mathrm{m}}
\DeclareMathOperator{\condexp}{\mathbb{E}}
\newcommand{\condex}[2]{\condexp({#1}\vert{#2})}
\newcommand{\N}{\mathbb{N}}
\newcommand{\Z}{\mathbb{Z}}
\newcommand{\R}{\mathbb{R}}

\newcommand{\PSL}{\mathrm{PSL}}
\newcommand{\SO}{\mathrm{SO}}
\renewcommand{\emptyset}{\varnothing}
\newcommand{\SetRec}{\mathcal{R}}
\newcommand{\Syndetic}{\mathcal{S}}
\newcommand{\Thick}{\mathcal{T}}
\newcommand{\PWSyndetic}{\mathcal{P}}
\newcommand{\idG}{\mathrm{id}_{G}}
\newcommand{\ind}[1]{\mathbf{1}_{#1}} % Indicator function
\def\<{\left\langle}
\def\>{\right\rangle}

\DeclareMathOperator{\Cb}{C_{b}}
\DeclareMathOperator*{\ap}{AP}

\DeclareMathOperator*{\clim}{C-lim}
\DeclareMathOperator*{\dlim}{D-lim}
\DeclareMathOperator*{\uclim}{UC-lim}
\DeclareMathOperator*{\udlim}{UD-lim}
\DeclareMathOperator{\cont}{C}
\DeclareMathOperator{\contc}{C_c}
\DeclareMathOperator{\conv}{\star}
\DeclareMathOperator{\dec}{D}
\DeclareMathOperator{\dens}{d}

\DeclareMathOperator{\inc}{I}
\DeclareMathOperator{\inv}{\mathscr{I}}

\DeclareMathOperator{\kron}{\mathscr{K}}
\DeclareMathOperator{\wm}{\mathscr{W}}
\DeclareMathOperator{\lowerdens}{\underline{\dens}}

\DeclareMathOperator{\lp}{L}

\DeclareMathOperator{\ret}{R}

\DeclareMathOperator{\symdiff}{\triangle}

\DeclareMathOperator{\upperdens}{\overline{\dens}}
\DeclareMathOperator{\upbdens}{d^*}
\DeclareMathOperator{\fs}{FS}
\DeclareMathOperator{\fp}{FP}
\DeclareMathOperator{\fpi}{FPI}
\DeclareMathOperator{\fpd}{FPD}

\begin{document}

\title{Finite Products Sets and Minimally Almost Periodic Groups}
\author{V.~Bergelson}
\email{vitaly@math.ohio-state.edu}
\thanks{The first author gratefully acknowledges the support of the NSF under grant DMS-1162073.}
\author{C.~Christopherson}
\email{cory@math.ohio-state.edu}
\author{D.~Robertson}
\email{robertson@math.ohio-state.edu}
\address{Department of Mathematics\\
  The Ohio State University\\
  231 West 18th Avenue\\
  Columbus\\
  OH 43210-1174\\
  USA}
\author{P.~Zorin-Kranich}
\email{pzorin@math.huji.ac.il}
\address{Institute of Mathematics\\
  Hebrew University\\
  Givat Ram\\
  Jerusalem, 91904\\
  Israel}
\date{\today}

\begin{abstract}
We construct, in \lcsca{} groups, sets with large density that fail to have certain combinatorial properties.
For the property of being a shift of a set of measurable recurrence we show that this is possible when the group does not have cocompact von Neumann kernel.
For the stronger property of being piecewise-syndetic we show that this is always possible.
%, and produce from such sets dynamical systems that have no non-trivial minimal sub-systems, but have non-trivial invariant measures.

For minimally almost periodic, \lcsca{} groups, we prove that any left dilation of a positive density set by an open neighborhood of the identity contains a set of measurable recurrence, and that the same result holds, up to a shift, when the von Neumann kernel is cocompact.
This leads to a trichotomy for \lcsca{} group based on combinatorial properties of large sets.
We also prove, using a two-sided Furstenberg correspondence principle, that any two-sided dilation of a positive density set contains a two-sided finite products set.
\end{abstract}

\maketitle

\section{Introduction}
\label{sec:Introduction}

Given a sequence $x_n$ in $\N = \{1,2,\dots \}$ the \define{finite sums set} or \define{IP set} generated by $x_n$ is defined by
\begin{equation*}
\fs(x_n) = \bigg\{ \sum_{n \in \alpha} x_n \,:\, \emptyset \ne \alpha \subset \N \textrm{ finite} \bigg\}.
\end{equation*}
The following theorem, proved by N.~Hindman in 1974, confirmed a conjecture of Graham and Rothschild.
\begin{theorem}[Hindman, \cite{MR0349574}]
For any finite partition $\N = C_1 \cup \cdots \cup C_r$ there is some $i \in \{1,\dots,r\}$ and some sequence $x_n$ in $\N$ such that $\fs(x_n) \subset C_i$.
\end{theorem}
Given a partition result, it is natural to ask whether it has a density version.
For example, van der Waerden's theorem states that, for any finite partition of $\N$ there is a cell of the partition containing arbitrarily long arithmetic progressions.
The density version of van der Waerden's theorem, \szemeredi{}'s theorem, states that if $E \subset \N$ has \define{positive upper density}, meaning that
\begin{equation*}
\upperdens(E) = \limsup_{N \to \infty} \frac{|E \cap \{1,\dots, N\}|}{N}
\end{equation*}
is positive, then $E$ contains arbitrarily long arithmetic progressions.
If the above limit exists then its value is called the \define{density} of $E$ and denoted $\dens(E)$.

In an attempt to discover a density version of Hindman's result, \erdos{} asked whether $\dens(E) > 0$ implies that $E$ contains a shift of some finite sums set.
In other words, does every $E \subset \N$ satisfying $\dens(E) > 0$ contain a set of the form $\fs(x_n) + t$ for some sequence $x_n$ in $\N$ and some $t$ in $\N$? It is necessary to allow for a shift because having positive density is a shift-invariant property, whereas being a finite sums set is not.
Indeed $2\N + 1$ has positive density but contains no finite sums set.
The following theorem of E.~Straus provided ``a counterexample to all such attempts'' (\cite[p.~105]{MR593525}).
\begin{theorem}[E.~Straus, unpublished; see {\cite[Theorem 2.20]{MR2259058}} or {\cite[Theorem 11.6]{MR564927}}]
\label{strausExample}
For every $\epsilon > 0$ there is a set $E \subset \N$ having density $\dens(E) > 1 - \epsilon$ such that no shift of $E$ contains a finite sums set.
\end{theorem}
This is proved by removing the tails of sparser and sparser infinite progressions from $\N$: for each $n$ in $\N$ one removes a set of the form $\{ a_n k + n \,:\, k \ge b_n \}$.
Since every finite sums set $\fs(x_n)$ intersects every set of the form $t \N$ (because infinitely many of the $x_n$ are congruent modulo $t$, say) after these removals one is left with a subset of $\N$ no shift of which contains a finite sums set.
Moreover, by carefully choosing the values of $a_n$ and $b_n$, one can ensure that the density of the resulting set exists and is as close to $1$ as we like.
Since any subset $F$ of $\N$ with $\upperdens(F) = 1$ contains arbitrarily long intervals, and hence a finite sums set, $\dens(E) > 1 - \epsilon$ is the best we can expect.

This paper is concerned with the natural task of finding and characterizing groups in which an analogue of Theorem \ref{strausExample} holds.
More specifically, we wish to describe those groups in which there are arbitrarily large sets no shift of which contains some sort of structured set.
To make this precise we need to decide what we mean by ``large'' and ``structured''.
We will consider only locally compact, Hausdorff topological groups, hereafter called simply ``locally compact groups''.
Within this class of groups, we can define a notion of largeness whenever the group has a \folner{} sequence.

\begin{definition}
Let $G$ be a locally compact group with a left Haar measure $\haar$.
A sequence $N \mapsto \Phi_N$ of compact, positive-measure subsets of $G$ is called a \define{left \folner{} sequence} if
\begin{equation}
\label{eq:def-left-folner}
\lim_{N \to \infty} \frac{\haar(\Phi_N \cap g\Phi_N)}{\haar(\Phi_N)} = 1
\end{equation}
uniformly on compact subsets of $G$, and a \define{right \folner{} sequence} if
\begin{equation*}
\lim_{N \to \infty} \frac{\haar(\Phi_Ng \cap \Phi_N)}{\haar(\Phi_N)} = 1
\end{equation*}
uniformly on compact subsets of $G$.
By a \define{two-sided \folner{} sequence} we mean a sequence that is both a left \folner{} sequence and a right \folner{} sequence.
\end{definition}

A locally compact group $G$ with a left Haar measure $\haar$ is said to be \define{amenable} if $\lp^\infty(G,\haar)$ has a left-invariant mean.
Every \lcsca{} group has a left \folner{} sequence \cite[Theorem 4.16]{MR961261}.
This was originally proved by \folner{} for countable groups in \cite{MR0079220}.

%If $\Phi$ is a left \folner{} sequence in $G$, then $N \mapsto \Phi_N^{-1}$ is a right \folner{} sequence in $G$.
%One can show, using \cite[I.\textsection 1, Proposition 2]{MR910005}, that any countable, amenable group has a two-sided \folner{} sequence.

\begin{definition}
\label{def:density}
Let $\Phi$ be a left, right, or two-sided \folner{} sequence in a \lcsca{} group $G$ with a left Haar measure $\haar$ and let $E$ be a measurable subset of $G$.
Define the \define{upper density} of $E$ with respect to $\Phi$ to be
\begin{equation*}
\upperdens_\Phi(E) = \limsup_{N \to \infty} \frac{\haar(E \cap \Phi_N)}{\haar(\Phi_N)},
\end{equation*}
the \define{lower density} of $E$ with respect to $\Phi$ by
\begin{equation*}
\lowerdens_\Phi(E) = \liminf_{N \to \infty} \frac{\haar(E \cap \Phi_N)}{\haar(\Phi_N)},
\end{equation*}
and the \define{density} of $E$ with respect to $\Phi$ as $\dens_\Phi(E) = \upperdens_\Phi(E) = \lowerdens_\Phi(E)$ whenever the limit exists.
\end{definition}

Thus the question we wish to answer is the following one: in which \lcsca{} groups is it possible to find a subset with arbitrarily large density no shift of which contains a set of some specified type?
In this paper we concern ourselves with two types of sets: piecewise syndetic sets and sets of measurable recurrence.

\begin{definition}
Let $G$ be a topological group.
A subset $S$ of $G$ is called \define{syndetic} if there exists a compact set $K$ such that $KS=G$.
A subset $T$ of $G$ is called \define{thick} if for every compact set $K$ there exists $g\in G$ such that $Kg\subset T$.
A subset $P$ of $G$ is called \define{piecewise syndetic} if there exists a compact set $K$ such that $KP$ is thick.
\end{definition}

To be precise, we should speak of ``left thick'', ``left syndetic'' and ``left piecewise syndetic'' sets. However, since we will not need the corresponding right-sided notions, we will continue to omit ``left'' from the terminology.

\begin{definition}
\label{def:meas-action}
Let $(X,\mathscr{B},\mu)$ be a separable probability space.
By a \define{measure-preserving action} of a topological group $G$ on $(X,\mathscr{B},\mu)$ we mean a jointly measurable map $T : G \times X \to X$ such that the induced maps $T^g : X \to X$ preserve $\mu$ and satisfy $T^g T^h = T^{gh}$ for all $g,h$ in $G$.
\end{definition}

\begin{definition}
\label{def:SetRec}
A subset $R$ of a topological group $G$ is a set of \define{measurable recurrence} if for every compact set $K$ in $G$, every measure-preserving action $T$ of $G$ on a separable probability space $(X,\mathscr{B},\mu)$, and every non-null, measurable subset $A$ of $X$ there exists $g$ in $R \setminus K$ such that $\mu(A \cap T^g A) > 0$.
\end{definition}

For piecewise-syndetic sets we have a version of Theorem~\ref{strausExample} in every \lcsca{} group that is not compact.
The proof is given in Section~\ref{sec:Extras}.

\begin{theorem}
\label{thm:large-lower-non-PWS}
For any \lcsca{} group $G$ that is not compact, any left \folner{} sequence $\Phi$ in $G$, and any $\epsilon>0$ there is a closed subset $Q$ of $G$ with $\lowerdens_\Phi(Q) > 1 - \epsilon$ that is not piecewise syndetic.
\end{theorem}

Whether a version of Theorem~\ref{strausExample} holds for sets of measurable recurrence depends on how many finite-dimensional representations the group has.

\begin{definition}
By a \define{representation} of a topological group $G$ we mean a continuous homomorphism from $G$ to the unitary group of a complex Hilbert space.
A group is a \define{WM group} if it has no non-trivial finite-dimensional representations.
A group is a \define{virtually WM group} if it has a subgroup of finite index that is a WM group.
Lastly, a group is \define{WM-by-compact} if it has a closed, normal, WM subgroup that is cocompact.
\end{definition}

The terminology WM is from ergodic theory.
In \cite{MR2561208}, a topological group is said to be WM if any ergodic, measure preserving action is weakly mixing.
By Theorem~3.4 in \cite{MR735224}, this agrees with our terminology.
It follows immediately from von Neumann's work on almost-periodic functions (discussed in Section \ref{sec:Preliminaries}) that a group is a WM group if and only if it is minimally almost-periodic (which means that the only almost-periodic functions on the group are the constant functions).

One can measure how far from being WM a topological group $G$ is by considering the closed, normal subgroup $G_0$ of $G$, obtained by intersecting the kernels of all finite-dimensional representations: $G$ is WM if and only if $G_0 = G$.
We have a version of Theorem~\ref{strausExample} for sets of measurable recurrence whenever $G/G_0$ is not compact.
We give the proof in Section~\ref{sec:notvirtuallyWM}.

\begin{theorem}
\label{thm:straus-set}
Let $G$ be a \lcsca{} group such that $G/G_0$ is not compact.
Then for any left \folner{} sequence $\Phi$ in $G$ and any $\epsilon > 0$ there is a measurable subset $E$ of $G$ with $\dens_\Phi(E) > 1 - \epsilon$ such that no set of the form $KEK$ with $K\subset G$ compact contains a set of measurable recurrence.
\end{theorem}

When $G = \mathbb{Z}$ the group $G_0$ is trivial ($G_0$ is trivial for every countable, abelian group) so Theorem~\ref{thm:straus-set} strengthens Theorem~\ref{strausExample} by exhibiting a set $E$ with $\dens_\Phi(E) > 1 - \epsilon$ for any prescribed \folner{} sequence $\Phi$.
Since every finite sums set is a set of measurable recurrence (see Example~\ref{eg:fpiMeasRec}) it also strengthens Theorem~\ref{strausExample} by widening the class of sets that cannot be shifted into $E$.

As in the proof of Theorem \ref{strausExample}, we will prove Theorem~\ref{thm:straus-set} by removing sparser and sparser sets from $G$.
To prohibit shifts of sets of measurable recurrence, we will construct $E$ by removing from $G$ shifts of tails of return time sets arising from certain $G$ actions.
That is, for certain measure-preserving actions $T : G \times X \to X$ we will remove from $G$ shifts of tails of sets with small density of the form $\{ g \in G : T^g x \in U \}$.
These actions will have a unique invariant measure, which will imply that the density of the sets we remove will exist.
Combining this with a version of the monotone convergence theorem for density (Lemma~\ref{lem:union-cofinite}) will allow us to prove that the density $\dens_\Phi(E)$ exists.
It is only when $G/G_0$ is not compact that we can produce sufficiently many actions suitable for this approach.

The two-sided finite product sets, which we now define, are what prevent Theorem~\ref{thm:straus-set} from holding in the WM case.
\begin{definition}
\label{def:finProdSets}
Let $\mathscr{F} = \{ \alpha \subset \N : 0 < |\alpha| < \infty \}$.
For any sequence $g_n$ in a topological group $G$ such that $g_{n}\to\infty$ (in the sense that it eventually leaves any compact subset of $G$) and any $\alpha = \{ k_1 < \cdots < k_m \}$ in $\mathscr{F}$ define $\inc_\alpha(g_n) = g_{k_1} \cdots g_{k_m}$ and $\dec_\alpha(g_n) = g_{k_m} \cdots g_{k_1}$.
Then $\fpi(g_n) := \{ \inc_\alpha(g_n) : \alpha \in \mathscr{F} \}$ is the \define{increasing finite products set} determined by the sequence $g_n$ and $\fpd(g_n) := \{ \dec_\alpha(g_n) : \alpha \in \mathscr{F} \}$ is the \define{decreasing finite products set} determined by the sequence $g_n$.
Put $\inc_\emptyset(g_n) = \dec_\emptyset(g_n) = \idG$.
Lastly, define by
\begin{equation*}
\fp(g_n) = \{ \inc_\alpha(g_n) \dec_\beta(g_n) : \alpha, \beta \in \mathscr{F} \cup \{ \emptyset \}, \alpha \cap \beta = \emptyset, \alpha \cup \beta \ne \emptyset \}
\end{equation*}
the \define{two-sided finite products set} determined by the sequence $g_n$.
\end{definition}

We insist that the sequence $g_n$ determining a finite products set escapes to infinity in order to produce sets of measurable recurrence, cf.~Example~\ref{eg:fpiMeasRec}.
Note that $\fp(g_n)$ contains both $\fpi(g_n)$ and $\fpd(g_n)$ because we can take either $\beta$ or $\alpha$ to be empty.

\begin{definition}
Let $G$ be a \lcsca{} group.
A subset $S$ of $G$ is called \define{left substantial} if $S \supset UW$ for some non-empty, open subset $U$ of $G$ containing $\idG$ and some measurable subset $W$ of $G$ having positive upper density with respect to some left \folner{} sequence in $G$.
\end{definition}

By \cite[Theorem~2.4]{MR2561208}, a \lcsca{} group $G$ is WM if and only if every left substantial subset of $G$ contains an increasing finite products set.
It follows that Theorem~\ref{thm:straus-set} fails badly when $G$ is WM.

\begin{definition}
\label{def:substantial:unimod}
Let $G$ be a \lcsca{}, unimodular group.
We say that a subset $S$ of $G$ is \define{substantial} if $S \supset UWU$ for some non-empty, open subset $U$ of $G$ containing $\idG$ and some measurable subset $W$ of $G$ having positive upper density with respect to some two-sided \folner{} sequence in $G$.
\end{definition}

Our next result, a two-sided version of \cite[Theorem~2.4]{MR2561208}, strengthens the degree to which a version of Theorem~\ref{strausExample} is unavailable in WM groups.
(Note that WM groups are always unimodular by Lemma~\ref{lem:WMunimod}.)

\begin{theorem}
\label{mainTheorem:WM:lcsc}
Let $G$ be a \lcsca{}, WM group.
Then every substantial subset of $G$ contains a two-sided finite products set.
\end{theorem}

As in \cite{MR2561208}, we deduce Theorem~\ref{mainTheorem:WM:lcsc} from results about measurable recurrence of WM groups using a version of the Furstenberg correspondence principle.
Since two-sided finite products sets involve multiplication on the left and on the right, we will need a two-sided version of the correspondence principle.
This is developed in Section~\ref{sec:fc}.
Moreover, since we deal with left and right multiplication, we will need to handle multiple recurrence for these group actions.
Results from \cite{arxiv:1309.6095}, which are based on ideas introduced in \cite{MR2539560} and \cite{MR2599882}, will allow us to do this.

In both \cite[Theorem~2.4]{MR2561208} and Theorem~\ref{mainTheorem:WM:lcsc}, attention is restricted to substantial subsets because, as explained in Section~\ref{sec:fc} below, there is in general no correspondence principle for arbitrary positive-density subsets.

We will show (see Example~\ref{eg:twoSidedNecessary}) that it is not sufficient to assume the set in Theorem~\ref{mainTheorem:WM:lcsc} is left substantial.
In fact we will construct, in a countable group, a set that has density $1$ with respect to a left \folner{} sequence and yet cannot contain a decreasing finite products set.
By \cite[Theorem~2.4]{MR2561208} this set must have zero density with respect to every right \folner{} sequence.
The existence of such a set answers the question, raised in \cite{arxiv:1307.0767}, of whether a set having positive upper density with respect to some left \folner{} sequence must have positive upper density for some right \folner{} sequence.

Our next result, a consequence of Theorem~\ref{mainTheorem:WM:lcsc}, shows that a version of Theorem~\ref{strausExample} fails to hold whenever $G$ is WM-by-compact.

\begin{theorem}
\label{mainTheorem:virtuallyWM:lcsc}
Let $G$ be a \lcsca{} group that is WM-by-compact but not WM.
Then every substantial subset of $G$ contains a left-shift of some two-sided finite products set and a right-shift of some (possibly different) two-sided finite products set.
On the other hand, for every two-sided \folner{} sequence $\Phi$ in $G$ there exists a subset of $G$ having positive upper density with respect to $\Phi$ containing no set of measurable recurrence.
\end{theorem}

In view of the fact (see Section~\ref{subsec:apFunc} below) that a countable group $G$ is virtually WM if and only if $G_0$ has finite index, Theorems~\ref{thm:straus-set}, \ref{mainTheorem:WM:lcsc} and \ref{mainTheorem:virtuallyWM:lcsc} together yield the following trichotomy for countable, amenable groups.

\begin{theorem}
\label{thm:countableClassification}
Let $G$ be a countable, infinite, amenable group.
Then exactly one of the following holds.
\begin{enumerate}
\item $G$ is WM.
Then every subset of $G$ having positive upper density with respect to some two-sided \folner{} sequences contains a two-sided finite products set.
\item $G$ is virtually WM, but not WM.
Then every subset of $G$ having positive upper density with respect to some two-sided \folner{} sequences contains a left shift of some two-sided finite products set and a right shift of some two-sided finite products set.
However, for any two-sided \folner{} sequence there is a subset with positive upper density that does not contain a set of measurable recurrence.
\item $G$ is not virtually WM.
Then for any two-sided \folner{} sequence $\Phi$ and any $\epsilon > 0$ there is a subset $E$ of $G$ with $\dens_\Phi(E) > 1 - \epsilon$ such that no set of the form $KEK$, where $K \subset G$ is finite, contains a set of measurable recurrence.
\end{enumerate}
\end{theorem}

In general, however, the situation is more complicated.
As Example~\ref{ex:klaus-schmidt} shows, there are \lcsca{} groups $G$ which have cocompact von Neumann kernel and yet fail to be WM-by-compact.
Nevertheless, for such groups we still have a one-sided version of Theorem~\ref{mainTheorem:virtuallyWM:lcsc}.
It would be interesting to know whether a two-sided version holds.

\begin{theorem}
\label{thm:hindman-G0-cocompact}
Let $G$ be a non-compact, \lcsca{} group such that $G/G_{0}$ is compact.
Let $S$ be a substantial subset of $G$.
Then $S$ contains a (left or right) shift of an increasing finite products set.
\end{theorem}

Thus we have the following one-sided trichotomy for \lcsca{} groups that are not compact.

\begin{theorem}
\label{thm:lcscaClassification}
Let $G$ be a \lcsca{} group that is not compact.
Then exactly one of the following holds.
\begin{enumerate}
\item $G$ is WM.
Then every substantial subset of $G$ contains an increasing finite products set.
\item $G$ is not WM, but $G/G_0$ is compact.
Then every substantial subset of $G$ contains a left shift of an increasing finite products set and a right shift of an increasing finite products set.
However, there are substantial sets that do not contain a set of measurable recurrence.
\item
\label{thm:lcscaClassification:}
$G/G_0$ is not compact.
Then there are substantial subsets with density arbitrarily close to 1 no shift of which contains a set of measurable recurrence.
\end{enumerate}
\end{theorem}

The rest of the paper runs as follows.
In the next section we recall some necessary facts about locally compact groups, \folner{} conditions and almost periodic functions, and what we need about topological dynamics, sets of measurable recurrence, and certain classes of large subsets of groups.
In Section~\ref{sec:fc} we discuss a two-sided Furstenberg correspondence principle that relates ergodic theory to combinatorics.
Section~\ref{sec:WM} starts with the necessary background on magic extensions and the multiple recurrence results these imply.
After using these results to prove Theorem~\ref{mainTheorem:WM:lcsc}, we give an example to show that the result is not always true for left substantial sets.
We conclude the section with a proof of Theorem~\ref{mainTheorem:virtuallyWM:lcsc}.
In Section~\ref{sec:notvirtuallyWM} we prove Theorem~\ref{thm:straus-set} using a version of the monotone convergence theorem for density and describe some combinatorial properties of the sets the theorem produces.
Lastly, in Section \ref{sec:Extras}, we prove Theorem~\ref{thm:large-lower-non-PWS} and use this result to exhibit a topological dynamical system $(X,G)$ in which the only minimal closed, invariant subset of $X$ is a singleton, but having an invariant measure that is non-atomic.

We would like to thank Klaus Schmidt for pointing out Example~\ref{ex:klaus-schmidt}.

\section{Preliminaries}

Throughout this paper, we write ``locally compact group'' as a shorthand for ``locally compact, Hausdorff topological group''.

\label{sec:Preliminaries}
\subsection{Locally compact groups}
We will prove most of our results for \lcsca{} groups.
Working at this level of generality exposes more completely the differences between left multiplication and right multiplication.

Throughout this section $G$ denotes a locally compact, second countable group and $\haar$ is a fixed left Haar measure on $G$.
Given $g$ in $G$ and a function $f : G \to \mathbb{R}$ we denote by $\lmult_g f$ and $\rmult_g f$ the functions $G \to \mathbb{R}$ defined by $(\lmult_g f)(x) = f(g^{-1}x)$ and $(\rmult_g f)(x) = f(xg)$ respectively.
With the exception of these $\lmult$ and $\rmult$ being lowercase, we follow the notational conventions in \cite{MR1397028}.
For instance, $\| \cdot \|_\infty$ denotes the supremum norm on bounded functions, and given a function $f$ on $G$, we define $\check{f}$ by $\check{f}(g) = f(g^{-1})$ for all $g$ in $G$.

The \define{modular function} of $G$ is the unique continuous homomorphism $\modular : G \to (0,\infty)$ such that $\modular(g) \|\rmult_{g}f\|_{1}=\|f\|_{1}$ for every $g\in G$ and $f\in \lp^1(G,\haar)$.
Recall that $G$ is said to be \define{unimodular} if $\modular = 1$.
We record the following for completeness.

\begin{lemma}
\label{lem:WMunimod}
Any locally compact, WM group $G$ is unimodular.
\end{lemma}
\begin{proof}
Composing the modular function $\modular$ of $G$ with any character of $(0,\infty)$ gives a representation of $G$ on $\mathbb{C}$.
Since $G$ is WM this representation is trivial, so the modular function takes values in the kernel of the character.
Choosing characters whose kernels intersect trivially shows that $\modular = 1$.
\end{proof}

The space $\lp^1(G,\haar)$ becomes an involutive Banach algebra upon defining an isometric involution $f \mapsto f^*$ by $f^{*}(x) = \overline{f(x^{-1})} \modular(x^{-1})$ and taking the convolution
\begin{equation*}
(f \conv h)(x) = \int f(y) h(y^{-1} x) \intd\haar(y) = \int f(x y) h(y^{-1}) \intd\haar(y)
\end{equation*}
as the multiplication.
We have $\lmult_{g}(f \conv h) = (\lmult_{g}f) \conv h$ and $\rmult_{g}(f \conv h)=f \conv (\rmult_{g}h)$ for any $g\in G$ and any functions $f,h$ in $\lp^1(G,\haar)$.

\begin{definition}
A function $f : G \to \mathbb{R}$ is \define{left uniformly continuous} if, for any $\epsilon > 0$, there is a neighborhood $V$ of $\idG$ in $G$ such that $\|\lmult_{v}f - f\|_{\infty} < \epsilon$ for all $v$ in $V$ and \define{right uniformly continuous} if, for any $\epsilon > 0$, there is a neighborhood $V$ of $\idG$ in $G$ such that $\|\rmult_{v}f - f\|_{\infty} < \epsilon$ for all $v$ in $V$.
A function is \define{uniformly continuous} if it is both left and right uniformly continuous.
\end{definition}

Note that the space of uniformly continuous functions is closed in the supremum norm.
Since $l$ is a continous action of $G$ on $\lp^{1}(G,\haar)$, the convolution $f\conv h$ is left uniformly continuous and the convolution $h\conv \check{f}$ is right uniformly continuous whenever $f\in\lp^{1}(G,\haar)$ and $h\in\lp^{\infty}(G,\haar)$.

We conclude the discussion of general \lcsc{} groups with the following basic fact that we will use often.
\begin{lemma}
\label{lem:lcsc-union-of-compact}
Let $G$ be a \lcsc{} group.
Then there exists an increasing sequence $K_{1}\subset K_{2}\subset \cdots$ of compact subsets of $G$ such that every compact subset of $G$ is contained in some $K_N$.
In particular $G= \cup \{ K_N : N \in \N \}$.
\end{lemma}
\begin{proof}
Let $g_{1},g_{2},\dots$ be a dense sequence in $G$ and let $U$ be a compact neighborhood of $\idG$.
Let $K_{N} = Ug_1 \cup \cdots \cup Ug_N$ for each $N$ in $\N$.
Let $V$ be the interior of $U$.
For any compact subset $K$ of $G$ we can find $N$ in $\N$ such that $Vg_1 \cup \cdots \cup Vg_N \supset K$.
Thus $K \subset K_N$.
Since $G$ is locally compact the sets $K_N$ cover $G$.
\end{proof}

\subsection{\folner{} conditions}
\label{sec:folnerConditions}
We will make use of two \folner{} conditions: \folner{} sequences and Reiter sequences.

We remark that it suffices to consider countably many compact sets $K$ when proving uniformity in \eqref{eq:def-left-folner}.
This follows from Lemma~\ref{lem:lcsc-union-of-compact}.
In fact, it even suffices to verify \eqref{eq:def-left-folner} for almost every $g$, since the local uniformity then follows from Egorov's theorem and continuity of convolutions. 

Every locally compact, second countable, amenable group has a left \folner{} sequence \cite[Theorem 4.16]{MR961261}.
Since we fixed a \emph{left} Haar measure $\haar$ to begin with, it is not immediately clear that such groups have right \folner{} sequences.

\begin{lemma}
Let $G$ be a locally compact, second countable, amenable group with a left Haar measure $\haar$.
Then $G$ has a right \folner{} sequence if and only if $G$ is unimodular.
\end{lemma}
\begin{proof}
When $G$ is unimodular the right Haar measure $\tilde{\haar}$ defined by $\tilde{\haar}(E) = \haar(E^{-1})$ agrees with $\haar$ on Borel sets.
If follows from this that $N \mapsto \Phi_N^{-1}$ is a right \folner{} sequence whenever $\Phi$ is a left \folner{} sequence.
On the other hand, if $G$ is not unimodular then $\modular$ is unbounded so there is some $g$ in $G$ with $\modular(g) \ge 3$.
Thus $\haar(E g \symdiff E) \ge 2\haar(E)$ for all Borel sets $E$, precluding the existance of a right \folner{} sequence.
\end{proof}

Thus it is impossible to find a two-sided \folner{} sequence in a general locally compact, second countable, amenable group.
However, by relaxing the requirement that $\Phi$ be a sequence of sets we can overcome this problem.

\begin{definition}
A sequence $\Phi$ of non-negative functions in $\lp^1(G,\haar)$ each having integral $1$ is called a \define{left Reiter sequence} if
\begin{equation}
\label{eqn:leftReiterDef}
\| \lmult_{g}\Phi_{N} - \Phi_{N}\|_{1} \to 0
\end{equation}
for every $g\in G$, a \define{right Reiter sequence} if
\begin{equation}
\label{eqn:reitReiterDef}
\|\modular(g)\rmult_{g}\Phi_{N} - \Phi_{N}\|_{1} \to 0
\end{equation}
for every $g\in G$, and a \define{two-sided Reiter sequence} if it is both left and right Reiter.
\end{definition}

Note that if $\Phi$ is a left \folner{} sequence in a locally compact, second countable group $G$ then
\begin{equation}
\label{eqn:folnerToAsymptotic}
N \mapsto \frac{1}{\haar(\Phi_N)} \cdot \ind{\Phi_N}
\end{equation}
is a left Reiter sequence.
Thus in particular every locally compact, second countable, amenable group $G$ has a left Reiter sequence.
Moreover, given a left \folner{} sequence $\Phi$ the sequence
\begin{equation*}
N \mapsto \frac{1}{\haar(\Phi_N)^2} \cdot \ind{\Phi_N} \conv \ind{\Phi_N}^*
\end{equation*}
is a two-sided Reiter sequence, so every locally compact, second countable, amenable group has such a sequence.
Note also that, when $G$ is unimodular, every left/right/two-sided \folner{} sequence $\Phi$ induces, via \eqref{eqn:folnerToAsymptotic}, a left/right/two-sided Reiter sequence.

Reiter sequences give rise to a notion of density that agrees with the usual notion of density when the sequence happens to arise from a \folner{} sequence as in \eqref{eqn:folnerToAsymptotic}.

\begin{definition}
Let $\Phi$ be a left/right/two-sided Reiter sequence in $\lp^1(G,\haar)$.
Given a Borel subset $E$ of $G$, denote by
\begin{equation*}
\upperdens_\Phi(E) = \limsup_{N \to \infty} \int \ind{E} \cdot \Phi_N \intd\haar
\end{equation*}
the \define{upper density} of $E$ with respect to $\Phi$.
The \define{lower density} of $E$ with respect to $\Phi$, which is denoted $\lowerdens_\Phi(E)$, is defined as the corresponding $\liminf$, and the \define{density} of $E$ with respect to $\Phi$ is the value $\dens_\Phi(E)$ of the limit, if it exists.
\end{definition}

In unimodular groups, the standard slicing argument allows one to construct a two-sided \folner{} sequence from a given two-sided Reiter sequence.
We will make use of a relativized version of this construction, in which the resulting \folner{} sequence assigns to a given subset a lower density not undercutting the upper density with respect to the given two-sided Reiter sequence.
\begin{proposition}
\label{prop:slicing}
Let $G$ be a locally compact, second countable, amenable group that is unimodular, and let $\Phi$ be a two-sided Reiter sequence in $\lp^1(G,\haar)$.
Then for every Borel subset $E\subset G$ there exists a two-sided \folner{} sequence $\Psi$ such that $\lowerdens_{\Psi}(E) \geq \upperdens_{\Phi}(E)$.
\end{proposition}
This result implies in particular that every \lcsc{} unimodular amenable group admits a two-sided \folner{} sequence.
For a shorter proof of this fact see \cite[I.\textsection 1, Proposition 2]{MR910005}.

\begin{proof}[Proof of Proposition~\ref{prop:slicing}]
Let $r=\upperdens_{\Phi}(E)$.
Passing to a subsequence we may assume that $\dens_{\Phi}(E)=r$.

\begin{claim}
Let $K\subset G$ be a compact set and $\epsilon>0$.
Then there exists a subset $K'\subset K$ with $\haar(K\setminus K')<\epsilon$ and a nonnull compact set $F$ in $G$ such that $\haar(F\symdiff xF)/\haar(F) < \epsilon$ and $\haar(Fx\symdiff F)/\haar(F) < \epsilon$ for every $x\in K'$ and $\haar(E\cap F)/\haar(F) > r - \epsilon$.
\end{claim}

We have $\Phi_{N} = \int_{0}^{\infty} \ind{A_{N,h}} \intd h = \int_{0}^{\infty} \chi_{N,h} \intd\lambda_{N}(h)$, where $A_{N,h}=\{\Phi_{N}>h\}$ are the superlevel sets, $\chi_{N,h} = \haar(A_{N,h})^{-1} \ind{A_{N,h}}$ are the normalized characteristic functions, and $\lambda_{N}$ is the probability measure on $(0,\infty)$ with density $h\mapsto \haar(A_{N,h})$ with respect to the Lebesgue measure.
We have
\[
r-\epsilon
<
\int \Phi_{N} \ind{E} \intd\haar
=
\int_{0}^{\infty} \int_{E} \chi_{N,h} \intd \haar \intd \lambda_{N}(h)
\]
for sufficiently large $N$.
It follows that the set $H = \{ h : \int_E \chi_{N,h} \intd\haar > r - 2\epsilon \}$ satisfies $\lambda_N(H) \ge \frac{\epsilon}{1-(r-2\epsilon)}$.

Consider
\[
| \lmult_{x}\Phi_{N} - \Phi_{N}| = \Big| \int_{0}^{\infty} \ind{x A_{N,h}} - \ind{A_{N,h}} \intd h \Big|.
\]
The central observation in the slicing argument is that the sets $A_{N,h}$ are nested, so the integrand on the right-hand side cannot take both strictly positive and strictly negative values at any given point.
Therefore the right-hand side equals
\[
\int_{0}^{\infty} | \ind{x A_{N,h}} - \ind{A_{N,h}} | \intd h.
\]
It follows that
\[
\| \lmult_{x}\Phi_{N} - \Phi_{N}\|_{1}
=
\int_{0}^{\infty} \haar(x A_{N,h} \symdiff A_{N,h}) \intd h
=
\int_{0}^{\infty} \frac{\haar(x A_{N,h} \symdiff A_{N,h})}{\haar(A_{N,h})} \intd \lambda_{N}(h).
\]
If $N$ is sufficiently large, then by \eqref{eqn:leftReiterDef} and the dominated convergence theorem
\begin{equation*}
\begin{aligned}
\frac{\epsilon}{1-(r-2\epsilon)}\frac{\epsilon^{2}}{2}
&
>
\int_{K}\|\lmult_{x}\Phi_{N} - \Phi_{N}\|_1 \intd \haar(x)
\\
&
=
\int_{0}^{\infty} \int_{K} \frac{\haar(x A_{N,h} \symdiff A_{N,h})}{\haar(A_{N,h})} \intd \haar(x) \intd \lambda_{N}(h).
\end{aligned}
\end{equation*}
Analogously, by \eqref{eqn:reitReiterDef}, the dominated convergence theorem and, crucially, the fact that $G$ is unimodular, for sufficiently large $N$ we have
\[
\frac{\epsilon}{1-(r-2\epsilon)}\frac{\epsilon^{2}}{2}
>
\int_{0}^{\infty} \int_{K} \frac{\haar(A_{N,h}x \symdiff A_{N,h})}{\haar(A_{N,h})} \intd \haar(x) \intd \lambda_{N}(h).
\]
It follows that
\[
\int_{K} \frac{\haar(x A_{N,h} \symdiff A_{N,h})}{\haar(A_{N,h})} + \frac{\haar(A_{N,h}x \symdiff A_{N,h})}{\haar(A_{N,h})} \intd \haar(x) < \epsilon^{2}
\]
for some $h\in H$.
Therefore $\frac{\haar(x A_{N,h} \symdiff A_{N,h})}{\haar(A_{N,h})} + \frac{\haar(A_{N,h}x \symdiff A_{N,h})}{\haar(A_{N,h})} < \epsilon$ for $x$ in a subset $K'\subset K$ with $\haar(K\setminus K') < \epsilon$, proving the claim.

\begin{claim}
Let $K\subset G$ be a compact set and $\epsilon>0$.
Then there exists a nonnull compact set $F$ in $G$ such that $\haar(xF\symdiff F)/\haar(F) < \epsilon$ and $\haar(Fx\symdiff F)/\haar(F) < \epsilon$ for every $x\in K$ and $\haar(E\cap F)/\haar(F) > r - \epsilon$.
\end{claim}
We may assume that $\haar(K)>\epsilon$.
We apply the previous claim to the set $\tilde K := K\cup KK$ with $\epsilon/2$ in place of $\epsilon$.
We obtain a subset $K' \subset \tilde K$ with $\haar(K\setminus K')<\epsilon/2$ and a compact set $F$ in $G$ such that $\haar(xF\symdiff F)/\haar(F) < \epsilon/2$ and $\haar(Fx\symdiff F)/\haar(F) < \epsilon/2$ for every $x\in K'$ and $\haar(E\cap F)/\haar(F) > r - \epsilon$.

Let now $k\in K$.
Then we have $k\tilde K \cap \tilde K \supseteq kK$, so that $\haar(k\tilde K \cap \tilde K) > \epsilon$.
It follows that $k K' \cap K' \neq \emptyset$, so that $k=k_{1}k_{2}^{-1}$ for some $k_{1},k_{2}\in K'$.
Therefore
\begin{multline*}
\haar(kF \symdiff F)
= \haar(k_{2}^{-1}F \symdiff k_{1}^{-1}F)
\leq \haar(k_{2}^{-1}F \symdiff F) + \haar(F \symdiff k_{1}^{-1}F)\\
\leq \haar(F \symdiff k_{2}F) + \haar(k_{1}F \symdiff F)
< \epsilon
\end{multline*}
and, since $G$ is unimodular,
\[
\haar(Fk \symdiff F)
= \haar(Fk_{1} \symdiff Fk_{2})
\leq \haar(Fk_{1} \symdiff F) + \haar(F \symdiff Fk_{2})
< \epsilon.
\]
This proves the claim.

The conclusion of the theorem follows quickly using the observation that it suffices to verify \eqref{eq:def-left-folner} for the increasing sequence of compact sets $K_{1}\subset K_{2} \subset \cdots$ given by Lemma~\ref{lem:lcsc-union-of-compact}: let the sets $\Psi_{N}$ be given by the last claim with the compact sets $K_{N}$ and $\epsilon=2^{-N}$. 
\end{proof}

\begin{definition}
Let $G$ be a locally compact, second countable, amenable group.
Define by
\begin{equation*}
\upbdens(E) = \sup \big\{ \upperdens_\Phi(E) : \Phi \textrm{ a two-sided Reiter sequence} \big\}
\end{equation*}
the \define{two-sided upper Banach density} of a Borel subset $E$ of $G$.
\end{definition}

Note that, by Proposition~\ref{prop:slicing}, when $G$ is unimodular the definition of two-sided upper Banach density is unchanged if one considers the supremum over only the two-sided \folner{} sequences.

\subsection{Almost periodic functions}
\label{subsec:apFunc}
We will now recall the notion of an almost-periodic function on a \lcsc{} group $G$ and its relationship with the finite-dimensional representations of $G$.
Denote by $\Cb(G)$ the Banach space of all bounded continuous functions $f : G \to \mathbb{C}$ equipped with the supremum norm.
The $G$-actions $\lmult$ and $\rmult$ on $\Cb(G)$ are isometric.
A function $f \in \Cb(G)$ is called \define{almost periodic} if one of the following equivalent conditions holds.
\begin{enumerate}
\item\label{def:ap:lmult} The subset $\{ \lmult_g f \,:\, g \in G \}$ of $\Cb(G)$ is relatively compact,
\item\label{def:ap:rmult} the subset $\{ \rmult_g f \,:\, g \in G \}$ of $\Cb(G)$ is relatively compact, or
\item\label{def:ap:compactification} $f$ is the pullback of a continuous function under the maximal topological group compactification $\iota : G \to \bar{G}$ (also called the \define{almost periodic} or \define{Bohr} compactification).
\end{enumerate}
For the equivalence \eqref{def:ap:lmult} $\iff$ \eqref{def:ap:rmult} see \cite[Theorem 9.2]{MR513591} and for the equivalence \eqref{def:ap:rmult} $\iff$ \eqref{def:ap:compactification} see \cite[Remark 9.8]{MR513591}.
Denote by $\ap(G)$ the space of all almost periodic functions on $G$.
It follows from the characterization \eqref{def:ap:compactification} that almost periodic functions are uniformly continuous.

The \define{matrix coefficients} of a finite-dimensional continuous representation give rise to almost-periodic functions on $G$.
Specifically, given a continuous representation $\phi$ of $G$ on a finite-dimensional, complex Hilbert space $V$ and vectors $x,y$ in $V$ we can form the almost-periodic function $f(g) = \langle \phi(g)x, y \rangle$.
\begin{theorem}[{\cite[Theorems 30 and 31]{0009.34902}}]
\label{thm:apCharacterization}
Matrix coefficients span a dense subspace of $\ap(G)$.
\end{theorem}

The constant functions are always almost-periodic.
There are groups having no other almost-periodic functions (see \cite{MR0002891} and Examples \ref{ex:wm:alt}, \ref{ex:wm:loc-finite}, and \ref{ex:wm:fin-gen}).
Such groups are said to be \define{minimally almost-periodic}.
Theorem \ref{thm:apCharacterization} implies that a group is minimally almost-periodic if and only if it is a WM group.
In view of the Peter--Weyl theorem, non-trivial compact groups are never WM groups.

\begin{example}
\label{ex:wm:alt}
Let $G$ be a countably infinite group that is the union of an increasing sequence of non-cyclic finite simple subgroups.
Then $G$ is WM.
This follows from \cite[Theorem~4.7]{MR2410393}, which states that the minimal dimension of a non-trivial representation of a non-cyclic finite simple group tends to infinity as the size of the group goes to infinity.
Moreover, $G$ is amenable since it is locally finite.

This applies, for instance, to the finite alternating group of the integers $A(\N)$, which is the subgroup of the finite symmetric group of the integers
\begin{equation*}
S(\N) = \{ \sigma : \N \to \N \,:\, \sigma \textrm{ is a bijection and } \{ n \in \N \,:\, \sigma(n) \ne n \} \textrm{ is finite} \}
\end{equation*}
consisting of the even permutations.
Indeed, $A(\N)$ is the union $\cup A_n$ where $A_n$ is the alternating group on $n$ points.
Now $[S(\N):A(\N)] = 2$, so $S(\N)$ is a virtually WM group.
On the other hand, $\phi(\sigma) = (-1)^{\mathrm{sgn}(\sigma)}$ defines a non-trivial representation of $S(\N)$, so $S(\N)$ is not a WM group.

This also applies to the projective linear group $\PSL_{n}(F)$ over an infinite algebraic field $F$ of finite characteristic for any $n\geq 2$.
Indeed, $F$ can be written as an increasing union of finite subfields $F^{(k)}$, and $\PSL_{n}(F)$ is the increasing union of copies of $\PSL_{n}(F^{(k)})$.
It is a classic result that $\PSL_{n}(F^{(k)})$ is simple unless $n=2$ and $|F^{(k)}|=2,3$.
\end{example}

\begin{example}
\label{ex:wm:loc-finite}
Recall that a group is called \define{periodic} if each of its elements has finite order.
By \cite[Theorem 2]{MR0003420} a periodic group admits a faithful finite-dimensional representation over a field of characteristic zero if and only if it has an abelian finite index normal subgroup that admits a faithful finite-dimensional representation.
It follows that any countably infinite periodic simple group $G$ is WM.
Indeed, suppose that $\pi$ is a non-trivial finite-dimensional representation of $G$.
Since $G$ is simple and $\pi$ is non-trivial, $\pi$ is faithful.
By the above characterization of faithfully representable groups $G$ admits a finite index normal abelian subgroup, which is a contradiction.

In particular this shows that countably infinite, locally finite, simple groups are WM.
Local finiteness is a stronger condition than periodicity, and it ensures amenability.
\end{example}

\begin{example}
\label{ex:wm:fin-gen}
Recall that a group is called \define{residually finite} if its points are separated by the homomorphisms into finite groups.
By \cite[Theorem 7]{MR0003420} any finitely generated group that admits a faithful representation over a field of characteristic zero is residually finite.
It follows that any countably infinite, finitely generated, simple group $G$ is WM.
Indeed, supppose that $G$ is not WM.
Since $G$ is simple, this implies that $G$ has a faithful finite-dimensional representation.
By the result cited above this forces $G$ to be residually finite, which contradicts simplicity.

It has been recently shown that there exist (many) countably infinite, finitely generated, simple, amenable groups, see \cite[Corollary B]{MR3071509}.
\end{example}

It follows from Theorem~\ref{thm:apCharacterization} that the subset
\begin{equation*}
\{ g \in G : f(g) = f(\idG) \textrm{ for all } f \in \ap(G) \}
\end{equation*}
of $G$, sometimes called the \define{von Neumann kernel} of $G$, is precisely $G_0$.
In fact, $G_{0}$ is the kernel of the almost periodic compactification $\iota$ of $G$.
Note, however, that $G_0$ need not be a WM group.

\begin{example}
Let
\begin{equation*}
G = \left\{ \begin{pmatrix} u & v\\0 & 1 \end{pmatrix} : u,v \in \mathbb{Q}, u \ne 0 \right\}
\end{equation*}
be the affine group of $\mathbb{Q}$.
It is shown in \cite{MR0002891} that
\begin{equation*}
G_0 = \left\{ \begin{pmatrix} 1 & v\\0 & 1 \end{pmatrix} : v \in \mathbb{Q} \right\}.
\end{equation*}
Clearly $G_0$, which is isomorphic to $(\mathbb{Q},+)$, is not WM.
\end{example}

The \lcsc{} groups $G$ can be classified according to the properties of their almost periodic compactification $\iota : G \to \bar{G}$ as follows.
\begin{enumerate}
\item The almost periodic compactification $\iota : G \to \bar{G}$ is not a surjective map.
Equivalently, $G/G_{0}$ is not compact.
Groups with this property are the subject of Theorem~\ref{thm:straus-set}.
In this case $G$ is not WM-by-compact, since for any WM subgroup $H\leq G$ we have $H\leq G_{0}$, and if $G/H$ is compact, then $G/G_{0}$ would also be compact.

\item The almost periodic compactification $\iota : G \to \bar{G}$ is surjective.
In this case our knowledge depends on the cardinality of $\bar{G}$.
\begin{enumerate}
\item $\bar{G}$ consists of one point, or equivalently $G=G_{0}$, or $G$ is WM.
In this case we have Theorem~\ref{mainTheorem:WM:lcsc}.
\item $\bar{G}$ consists of more than one point but is finite, or equivalently $1 < [G:G_{0}] < \infty$.
In this case $G_{0}$ is WM, since any non-trivial finite-dimensional representation of $G_{0}$ would induce a finite-dimensional representation of $G$ that does not vanish identically on $G_{0}$.
Corollary~\ref{mainTheorem:virtuallyWM:lcsc} holds.
\item $\bar{G}$ is infinite.
In this case $G_{0}$ need not be WM and, more generally, $G$ need not be WM-by-compact (see Example~\ref{ex:klaus-schmidt}).
Theorem~\ref{thm:hindman-G0-cocompact} holds.
Note that this case cannot occur for discrete $G$.
\end{enumerate}
\end{enumerate}

\begin{example}[K.~Schmidt]
\label{ex:klaus-schmidt}
The von Neumann kernel of a connected group admits an explicit description, see \cite{MR2918898}.
A special case of this description shows that the von Neumann kernel of a connected Lie group is the closed commutator subgroup (this can be seen as a version of Lie--Kolchin theorem, cf.~\cite[Theorem 3]{MR2394977}).
Thus the von Neumann kernel $G_{0}$ of the semidirect product $G=\SO(2) \ltimes \R^{2}$ (taken with respect to the defining action of $\SO(2)$ on $\R^{2}$) equals $\R^{2}$.
This shows that $G_{0}$ need not be WM when $G/G_{0}$ is compact in the case that $G$ is not discrete.
\end{example}

In order to prove Theorem~\ref{thm:straus-set} we will need a supply of almost periodic functions that vary sufficiently slowly.
\begin{lemma}
\label{lem:slowly-varying-ap}
Let $G$ be a locally compact, second countable, amenable group such that the almost periodic compactification $\iota : G \to \bar{G}$ is not surjective.
Then for every $\epsilon>0$ and every compact set $K$ in $G$ there exists an almost periodic function $f : G \to [0,1]$ such that $\|\lmult_{g}f-f\|_{\infty} < \epsilon$ for every $g\in K$ and the range of $f$ is $\epsilon$-dense in $[0,1]$.
\end{lemma}
\begin{proof}
Since $G$ is amenable, it has a left \folner{} sequence, so there exists a nonnull compact set $F\subset G$ such that $\haar(F \symdiff gF)/\haar(F) < \epsilon$ for every $g\in K$.
In particular, $\iota(F^{-1}) \subset \iota(G)$ is a compact subset.
On the other hand, by the assumption that $\iota$ is not surjective, $\iota(G)$ is a proper dense subgroup of the compact group $\bar{G}$, so it is not compact.
 
Therefore there exist $g_{i}\in G$, $i = 0,\dots, \lfloor 1/\epsilon \rfloor$, such that the sets $\iota(F^{-1}g_{i})$ are pairwise disjoint.
Let $\tilde f : \bar{G} \to [0,1]$ be a continuous function that equals $i\epsilon$ on $\iota(F^{-1}g_{i})$.
(Such a function can be constructed using the Urysohn lemma.)

Consider now the continuous function $f=\haar(F)^{-1} \int_{F} \lmult_{g}\tilde f \intd\haar(g)$ on $\bar{G}$.
Then $f(g_{i})=i\epsilon$, so that $f(G)$ is $\epsilon$-dense in $[0,1]$.
Moreover, it follows from the \folner{} condition that $\|\lmult_{g}f-f\|_{\infty}<\epsilon$ for every $g\in K$.
\end{proof}

\subsection{Topological dynamics}
We now recall some facts from topological dynamics.
A \define{topological dynamical system} $(X,G)$ is a compact metric space $(X,d)$ together with a jointly continuous left action $(g,x) \mapsto gx$ of $G$ on $(X,d)$.
We say that $(X,G)$ is \define{topologically transitive} if there is some $x \in X$ with dense orbit and \define{minimal} if every point in $X$ has dense orbit.
A system $(X,G)$ is \define{equicontinuous} if the collection of homeomorphisms $x\mapsto gx$, $g\in G$, is equicontinuous.
A Borel measure $\mu$ on $X$ is \define{invariant} with respect to a left action of $G$ if, for all Borel sets $A \subset X$ and all $g \in G$, one has $\mu(g A) = \mu(A)$.
A version of the Bogolioubov--Krylov theorem (see \cite[Corollary 6.9.1]{MR648108}) for amenable groups guarantees that if $G$ is amenable, then any topological dynamical system $(X,G)$ admits an invariant Borel probability measure.
The system $(X,G)$ is said to be \define{uniquely ergodic} if there is only one invariant Borel probability measure on $X$.
The \define{support} of a Borel probability measure on a compact metric space $(X,d)$ is the intersection of all closed sets with full measure.
Since $X$ is second countable, the support has full measure.
Given a topological dynamical system $(X,G)$, a point $x \in X$, and a subset $A$ of $X$, write $\ret_A(x) = \{ g \in G \,:\, gx \in A \}$.
Sets of the form $\ret_A(x)$ are called \define{return time sets}.

\begin{theorem}[{\cite[Theorem 7]{MR0033342}}]
\label{uniqueErgodicitySegal}
Let $(X,G)$ be a topological dynamical system that is topologically transitive and equicontinuous.
Then $(X,G)$ is minimal and uniquely ergodic.
If $X$ is infinite and $\mu$ is the unique $G$-invariant probability measure, then $\mu$ is non-atomic.
\end{theorem}

\begin{lemma}
\label{densityIsMeasure}
Let $(X,G)$ be a topological dynamical system.
If $(X,G)$ is uniquely ergodic with invariant Borel probability measure $\mu$, $A$ is a Borel set with $\mu(\partial A) = 0$, $x$ is a point in $X$, and $\Phi$ is any left Reiter sequence, then $\dens_\Phi(\ret_A(x)) = \mu(A)$.
\begin{proof}
Since $(X,G)$ is uniquely ergodic, for every $f$ in $\cont(X)$ and every $x \in X$ we have
\begin{equation*}
\lim_{N \to \infty} \int \Phi_N(g) f(gx) \intd\haar(g) = \int f \intd\mu.
\end{equation*}
Fix $\epsilon > 0$.
Since $\mu(\partial A) = 0$ we can find $f_1,f_2$ in $\cont(X)$ so that $f_1 \le \ind{A} \le f_2$ and $\int f_2 - f_1 \intd\mu < \epsilon$.
Thus
\begin{multline*}
\mu(A) - \epsilon \le \int f_1 \intd\mu
= \lim_{N \to \infty} \int \Phi_N(g) f_1(gx) \intd\haar(g)\\
\le \liminf_{N \to \infty} \int \Phi_N(g) \ind{A}(gx) \intd\haar(g)
\le \limsup_{N \to \infty} \int \Phi_N(g) \ind{A}(gx) \intd\haar(g)\\
\le \lim_{N \to \infty} \int \Phi_N(g) f_2(gx) \intd\haar(g)
= \int f_2 \intd\mu \le \mu(A) + \epsilon,
\end{multline*}
giving $\dens_\Phi(\ret_A(x)) = \mu(A)$ as desired.
\end{proof}
\end{lemma}

\begin{lemma}
\label{lem:smallZeroBoundaryBalls}
Let $(X,d)$ be a compact metric space with a non-atomic probability measure $\mu$.
For any $\epsilon > 0$ and any point $x$ in the support of $\mu$ there is an open set $A$ containing $x$ such that $\mu(A) < \epsilon$ and $\mu(\partial A) = 0$.
\begin{proof}
Let $x$ be a point in the support of $\mu$.
For every $t > 0$ the open ball $B_t$ centered at $x$ with radius $t$ has positive measure.
Their boundaries $\partial B_t$ are disjoint, so only countably many of the sets $\partial B_t$ have positive measure.
Let $t_n$ be a sequence decreasing to 0 such that $\mu(\partial B_{t_n}) = 0$ for all $n$.
We have $\mu(B_{t_n}) \to \mu(\{ x \}) = 0$ because $\mu$ is non-atomic.
Put $A = B_{t_n}$ with $n$ so large that $\mu(B_{t_n}) < \epsilon$.
\end{proof}
\end{lemma}

\subsection{Sets of measurable recurrence}
In this subsection we discuss sets of measurable recurrence in \lcsc{} groups.
Note that the joint measurability condition in Definition~\ref{def:meas-action} implies (see \cite[22.20(b)]{MR551496}, for example) that the induced action of $G$ on $\lp^2(X,\mathscr{B},\mu)$ is strongly continuous.
We begin by pointing out that, when $G$ is countable and infinite, Definition~\ref{def:SetRec} coincides with the usual definition.

\begin{proposition}
Let $G$ be a countable, infinite group.
Then a subset $R$ of $G$ is a set of measurable recurrence if and only if, for every measurable action of $G$ on a separable probability space $(X,\mathscr{B},\mu)$ and every $B$ in $\mathscr{B}$ with positive measure, we can find $r$ in $R \setminus \{ \idG \}$ such that $\mu(B \cap T^r B) > 0$.
\end{proposition}
\begin{proof}
It is clear that every set of measurable recurrence has the stated property.
Conversely, suppose that $G$ is a countable group and that $R$ is a subset of $G$ with the stated property.
Let $K$ be a compact subset of $G$, let $(X,\mathscr{B},\mu)$ be a separable probability space equipped with a measure-preserving $G$ action, and let $A$ be a nonnull measurable subset of $X$.
Consider the space $Y=\{0,1\}^{G}$ with the $(\frac12,\frac12)$-Bernoulli measure and the shift action of $G$.
Let $F =K\setminus\{\idG\}$, which is a finite set, and put
\begin{equation*}
B=\{y\in Y : y(\idG)=1, y(f)=0 \text{ for all } f\in F\}.
\end{equation*}
Then $A\times B$ is a nonnull measurable subset of $X\times Y$, and by the hypothesis there exists a $g\in R\setminus\{\idG\}$ such that $(A\times B)\cap (gA\times gB)$ is nonnull.
By construction of $B$ it follows that $g\in R\setminus K$ and $A\cap gA$ is nonnull.
Hence $R$ is a set of measurable recurrence.
\end{proof}

We remark that the measure of the set $B$ in the above proof can be improved considerably.
Indeed, above $B$ has measure $2^{-n-1}$, where $n=|F|$, but if one uses instead the $(\frac{n}{n+1},\frac{1}{n+1})$-Bernoulli measure then the measure of $B$ will become $\frac{1}{n+1}\big(\frac{n}{n+1}\big)^{n} > e^{-1}/(n+1)$.

\begin{example}
\label{eg:fpiMeasRec}
Let $G$ be a \lcsc{} group.
Let $P$ be a subset of $G$ such that for every neighborhood $U$ of $\idG$ the set $UP$ contains a set of the form $\fpi(g_n)$ with $g_n \to \infty$.
Then $P$ is a set of measurable recurrence.
\end{example}
\begin{proof}
Fix a compact set $K$ in $G$ and a measure-preserving action of $G$ on a separable probability space $(X,\mathscr{B},\mu)$.
Let $A$ be a non-null, positive-measure subset of $X$.
Since the induced action on $L^{2}(X,\mathscr{B},\mu)$ is strongly continuous, there is a relatively compact, symmetric neighborhood $U$ of the identity such that $\mu(A\cap gA) > \mu(A) - \mu(A)^{2}/2$ for any $g\in U$.

Let $n\mapsto g_{n}$ be a sequence such that $g_n \to \infty$ and $\fpi(g_{n}) \subset UP$.
Let $n \mapsto h_n$ be a subsequence of $n \mapsto g_n$ such that $\fpi(h_n) \subset G \setminus \overline{U}K$.
Such a sequence can be constructed by choosing $h_{n+1}$ from
\begin{equation*}
\{ g_i : i \in \N \} \setminus \cup \, \{ \inc_\alpha(h_i)^{-1} \overline{U}K : \emptyset \ne \alpha \subset \{ 1,\dots,n\} \}
\end{equation*}
inductively.
By a standard argument (originally due to Gillis \cite{MR1574762}) we have
\begin{equation*}
\mu(h_1 \cdots h_m A \cap h_1 \cdots h_n A) > \mu(A)^{2}/2
\end{equation*}
for some $m > n$.
Then $h_{n+1} \cdots h_m \in \fpi(g_n) \setminus \overline{U}K$ and $\mu(A \cap h_{n+1} \cdots h_m A) > \mu(A)^{2}/2$.
Let $g\in U$ be such that $gh_{n+1} \cdots h_m \in P$.
Then $\mu(A \cap g h_{n+1} \cdots h_m A) > 0$.
\end{proof}

In any given group $G$ we denote the class of sets of measurable recurrence by $\SetRec$.

\begin{lemma}
\label{lem:meas-rec-part-reg}
The class $\SetRec$ is partition regular.
\end{lemma}
\begin{proof}
Suppose $R_{1}\cup\dots\cup R_{r}\in\SetRec$.
We need to show that $R_{i}\in\SetRec$ for some $i\in\{1,\dots,r\}$.
Assume $R_{i}\not\in\SetRec$ for all $i=1,\dots,r$.
Thus for each $1 \le i \le r$ there is a compact set $K_{i}$ in $G$, a measure-preserving action $T_i$ of $G$ on a separable probability space, and a positive measure set $B_{i}$ in the probability space witnessing the fact that $R_{i}$ is not a set of measurable recurrence.
The positive-measure set $B_1 \times \cdots \times B_r$ in the product probability space equipped with the product action $T_1 \times \cdots \times T_r$, and the compact set $K_1 \cup \cdots \cup K_r$ now witness the fact that $R_1 \cup \cdots \cup R_r$ is not a set of measurable recurrence, which is a contradiction.
\end{proof}

Given a class of subsets it is natural (see \cite{MR2353901} for an extensive discussion) to consider its \define{dual class}, which consists of all subsets whose intersection with every member of the given class is non-empty.
Denote by $\SetRec^*$ the dual class of $\SetRec$.
It is clear that if $A\in\SetRec^{*}$, then every superset of $A$ is also a member $\SetRec^{*}$. 
It follows from partition regularity of $\SetRec$ that the intersection of any two members of $\SetRec^{*}$ is again a member of $\SetRec^{*}$.
Since every cocompact subset of a member of $\SetRec$ is clearly a member of $\SetRec$, this has the following consequence.
\begin{corollary}
\label{cor:cofin-pres-srs}
Any cocompact subset of an $\SetRec^{*}$ set is $\SetRec^{*}$.
\end{corollary}

We will need later the fact that return time sets in certain topological dynamical systems belong to $\SetRec^{*}$.

\begin{lemma}
\label{lem:top-ret-is-sec-ret}
Let $(X,G)$ be a minimal topological dynamical system where $G$ acts by isometries.
For any $x$ in $X$ and any neighborhood $U$ of $x$ the set $\ret_U(x)$ is in $\SetRec^{*}$.
\end{lemma}
\begin{proof}
Let $\delta > 0$ be such that the ball $B_\delta(x)$ is contained in $U$.
Let $\mu$ be an invariant Borel probability measure on $(X,G)$.
By minimality $V := B_{\delta/2}(x)$ has positive measure.
Since $G$ acts by isometries, the $\SetRec^*$ set $\{ g : V \cap gV \ne \emptyset \}$ is contained in $\ret_U(x)$.
\end{proof}

Finally, we note that the classes $\SetRec$ and $\SetRec^{*}$ are closed under conjugation.

\subsection{Syndetic, thick, and piecewise-syndetic sets}
Given a topological group $G$, denote by $\Syndetic$, $\Thick$, and $\PWSyndetic$, the classes of syndetic, thick, and piecewise syndetic subsets, respectively.

\begin{lemma}
\label{lem:synd-thick-dual}
$\Syndetic^{*}=\Thick$ and $\Thick^{*}=\Syndetic$.
\end{lemma}
\begin{proof}
First, let $S$ be syndetic in $G$ and let $T$ be thick in $G$.
By syndeticity there exists a compact set $K$ such that $KS=G$.
By thickness there exists $g\in G$ such that $K^{-1}g\subset T$.
Write $g=ks$ with $k\in K$, $s\in S$.
Thus the intersection of any syndetic set with any thick set is non-empty.
This implies $\Thick\subset\Syndetic^{*}$ and $\Syndetic\subset\Thick^{*}$.

We now prove that if $G = P \cup Q$ then either $P$ is thick or $Q$ is syndetic.
If $P$ is not thick then there exists a compact set $K$ such that for each $g \in G$ we have $Kg \cap Q \ne \emptyset$.
Thus for any $g \in G$ we can find $k \in K$ and $q \in Q$ such that $kg = q$.
This implies that $K^{-1}Q = G$ as desired.

Now, if $P$ does not belong to $\Thick$ then its complement is syndetic so $P$ does not belong to $\Syndetic^{*}$.
Similarly, if $Q$ does not belong to $\Syndetic$ then its complement is thick so $Q$ does not belong to $\Thick^{*}$.
Thus $\Thick\supset\Syndetic^{*}$ and $\Syndetic\supset\Thick^{*}$.
\end{proof}

\begin{lemma}
\label{lem:largeShiftsForThick}
Let $G$ be a group that is not compact.
Then for any thick subset $T$ of $G$ and any compact subset $K$ of $G$ there is $g \in G\setminus K$ such that $Kg \subset T$.
\end{lemma}
\begin{proof}
Let $T$ and $K$ be thick and compact subsets of $G$ respectively.
Let $h\in G\setminus KK^{-1}$.
By thickness we have $(K\cup Kh)g\subset T$ for some $g\in G$.
Suppose that $g\in K$ and $hg\in K$.
Then $h=(hg)g^{-1}\in KK^{-1}$, a contradiction.
Hence $g\not\in K$ or $hg\not\in K$, as required.
\end{proof}

\section{The Furstenberg correspondence principle}
\label{sec:fc}
Since the configurations we are interested in (see Definition~\ref{def:finProdSets}) are two-sided in nature, we need a correspondence principle that is sensitive to multiplication on the left and on the right.
Roughly speaking, a correspondence principle should relate recurrence for sets of positive upper density in our group to recurrence in a certain measure-preserving action of the group on a probability space.
It turns out that this is not possible for arbitrary Borel sets of positive upper density even in $\mathbb{R}$.
Indeed, setting $\Phi_N = [0,N]$ in $\mathbb{R}$, \cite[Theorem~D]{MR1269198} exhibits for all but countably many $\alpha > 0$ a subset $E$ of $\mathbb{R}$ with $\dens_\Phi(E) = 1/2$ such that $\dens_\Phi(E - n^\alpha \cap E) = 0$ for all $n$ in $\N$.
On the other hand, whenever $\alpha > 0$ is irrational, the set $\{ n^\alpha : n \in \N \}$ is a set of measurable recurrence because (as first proved in \cite{JFM56.0898.04}) the sequence $n^\alpha t$ is uniformly distributed $\bmod 1$ for any real, non-zero $t$.

An appropriate one-sided correspondence principle for general \lcsc{} groups was given in \cite[Theorem~1.1]{MR2561208}.
We now describe a two-sided version adequate for our needs.
It is for substantial subsets, which we now define in general, that we can prove a correspondence principle.

\begin{definition}
\label{def:substantial}
We say that a subset $S$ of $G$ is \define{substantial} if one can find a measurable subset $W$ of $G$ with $\upperdens_\Phi(W) > 0$ for some two-sided Reiter sequence $\Phi$ in $\lp^{1}(G,\haar)$ and a symmetric, open subset $U$ of $G$ containing $\idG$ such that $S \supset UWU$.
\end{definition}

In unimodular groups the notion of being substantial does not change if one demands $\upperdens_\Phi(W) > 0$ for some two-sided \folner{} sequence $\Phi$ in the definition; this follows from Proposition~\ref{prop:slicing}.
Thus Definition~\ref{def:substantial} agrees with Definition~\ref{def:substantial:unimod} when both apply.
%In discrete groups the above notion coincides with having positive upper Banach density.

\begin{theorem}[Correspondence principle]
\label{thm:lcscfcp}
Let $G$ be a \lcsca{} group with a left Haar measure $\haar$ and a two-sided Reiter sequence $\Phi$.
Let $W$ be a Borel subset of $G$, let $U$ be an open neighborhood of $\idG$, and let $S\supseteq UWU$ be measurable.
Then there is a compact metric space $X$ with a continuous $G\times G$-action $L\times R$ and a non-negative, continuous function $\xi$ on $X$ such that the following holds.
\begin{enumerate}
\item\label{thm:lcscfcp:mu}
There exists an $L\times R$-invariant Borel probability measure $\mu$ on $X$ such that $\int \xi \intd\mu = \upperdens_{\Phi}(W)$ and
\begin{equation*}
\upperdens_\Phi(g_1^{-1} S h_1 \cap \cdots \cap g_n^{-1} S h_n) \ge \int L^{g_1} R^{h_1} \xi \cdots L^{g_n} R^{h_n} \xi \intd\mu
\end{equation*}
for any $g_1,\dots,g_n,h_1,\dots,h_n$ in $G$.
\item\label{thm:lcscfcp:nu}
There exist an ergodic $L\times R$-invariant Borel probability measure $\nu$ on $X$ and a two-sided Reiter sequence $\Psi$ such that $\int \xi \intd\nu \geq \upperdens_{\Phi}(W)$ and
\begin{equation}
\label{eqn:lcscfcp:nu}
\upperdens_\Psi(g_1^{-1} S h_1 \cap \cdots \cap g_n^{-1} S h_n) \ge \int L^{g_1} R^{h_1} \xi \cdots L^{g_n} R^{h_n} \xi \intd\nu
\end{equation}
for any $g_1,\dots,g_n,h_1,\dots,h_n$ in $G$.
%\item\label{thm:lcscfcp:discrete}
%If $G$ is discrete, then $\xi$ is $\{0,1\}$-valued.
%\item\label{thm:lcscfcp:unimodular}
%If $G$ is unimodular, then
%\begin{equation*}
%\upbdens(g_1^{-1} S h_1 \cap \cdots \cap g_n^{-1} S h_n) \ge \int L^{g_1} R^{h_1} \xi \cdots L^{g_n} R^{h_n} \xi \intd\nu.
%\end{equation*}
\end{enumerate}
If $G$ is discrete, we can take $\xi$ to be a characteristic function of a clopen subset of $X$.
\end{theorem}

In the second part of Theorem~\ref{thm:lcscfcp} we obtain an ergodic measure-preserving system at the cost of modifying the Reiter sequence.
Thus if one is only interested in the two-sided upper Banach density, it suffices to consider ergodic measure-preserving systems, see Corollary~\ref{cor:fcp-erg:lcsc}.
This was already observed for the group $\mathbb{Z}$ in \cite[Proposition~3.1]{MR2138068}.

The probability space that our correspondence yields will be built on the Gelfand spectrum of a \cstar{}-algebra of functions on $G$.
Recall that one can think of the Gelfand spectrum as either the space of maximal ideals, or as the space of all non-trivial multiplicative linear forms.
As in \cite{MR2561208}, the \cstar{}-algebra we use will consist of uniformly continuous functions.

\begin{proof}[Proof of Theorem~\ref{thm:lcscfcp}]
We may assume without loss of generality that $U$ is symmetric.
Let $\psi$ be a continuous non-negative function supported in $U$ such that $\psi=\psi^{*}$ and $\int\psi \intd\haar = 1$.
The function $\xi := \psi \conv \ind{W} \conv \check{\psi}$ is non-negative, uniformly continuous, and dominated by $\ind{S}$.
We now consider the minimal closed $*$-subalgebra $\mathfrak{A}$ of bounded functions on $G$ that contains $\ind{G}$, contains $\xi$, and is invariant under $l$ and $r$.
By the Gelfand--Naimark theorem $\mathfrak{A}$ is canonically isomorphic to $C(X)$, where $X$ is the Gelfand spectrum of $\mathfrak{A}$, that is, the space of all non-trivial multiplicative linear forms on $\mathfrak{A}$ with the weak* topology.

Since $\mathfrak{A}$ consists of uniformly continuous functions, the $G\times G$-action $l\times r$ on $\mathfrak{A}$ is jointly continuous.
Since the action $l\times r$ is by C$^{*}$-algebra automorphisms, it follows that it induces a continuous $G\times G$ action $L\times R$ on $X$ such that $\lmult_{g}\rmult_{g} f = f \circ L_{g^{-1}}R_{g^{-1}}$ for every $f\in\mathfrak{A}$.
Since $\mathfrak{A}$ is separable, it follows that $X$ is metrizable.

Part \eqref{thm:lcscfcp:mu}.
Passing to a subsequence of $\Phi$ we may assume that $\dens_{\Phi}(W)$ exists.
Passing to a further subsequence of $\Phi$ we may assume that
\[
\mu(f) := \lim_{N\to\infty} \int \Phi_{N} f \intd\haar
\]
exists for every $f\in\mathfrak{A}$.
This is a positive unital linear functional on $\mathfrak{A}$, so by the Riesz--Markov--Kakutani representation theorem $\mu$ corresponds to a Borel probability measure on $X$.
The Reiter property of $\Phi$ implies that $\mu$ is $L\times R$-invariant.
We have
\begin{equation*}
\begin{aligned}
\mu(\xi)
 & =
\lim_{N \to \infty} \int \Phi_{N} (\psi \conv \ind{W} \conv \check{\psi}) \intd\haar\\
 & =
\lim_{N \to \infty} \iiint \Phi_{N}(x) \psi(y) \ind{W}(z) \check{\psi}(z^{-1}y^{-1}x) \intd\haar(x)\intd\haar(y)\intd\haar(z)\\
 & =
\lim_{N \to \infty} \iint \psi(y) \check{\psi}(x) \int \Phi_{N}(yzx) \ind{W}(z) \intd\haar(z)\intd\haar(x)\intd\haar(y)\\
 & =
\lim_{N \to \infty} \iint \psi(y) \psi(x) \modular(x) \int (\lmult_{y^{-1}}\rmult_{x}\Phi_{N})(z) \ind{W}(z) \intd\haar(z)\intd\haar(x)\intd\haar(y)\\
 & =
\lim_{N \to \infty} \int \Phi_{N}(z) \ind{W}(z) \intd\haar(z)
=
\dens_{\Phi}(W)
\end{aligned}
\end{equation*}
by the Reiter property of $\Phi$ and the dominated convergence theorem.

Lastly we have
\begin{equation*}
\begin{aligned}
&
\upperdens_\Phi(g_1^{-1} S h_1 \cap \cdots \cap g_n^{-1} S h_n)\\
\geq &
\lim_{N \to \infty} \int \Phi_{N} \cdot \lmult_{g_1^{-1}} \rmult_{h_1^{-1}} (\psi \conv \ind{W} \conv \check{\psi}) \cdots \lmult_{g_n^{-1}} \rmult_{h_n^{-1}} (\psi \conv \ind{W} \conv \check{\psi}) \intd\haar\\
= &
\int L^{g_1} R^{h_1} \xi \cdots L^{g_n} R^{h_n} \xi \intd\mu
\end{aligned}
\end{equation*}
for any $g_1,\dots,g_n,h_1,\dots,h_n$ in $G$ as desired.

Part \eqref{thm:lcscfcp:nu}.
Let $\nu$ be an ergodic component of $\mu$ under the action $L\times R$ such that $\int \xi \intd\nu \geq \int \xi \intd\mu$.
%Unfortunately, it does not seem to be possible to select the ergodic component in such a way that \eqref{eq:udcap-int} continues to hold.
Let $e:\mathfrak{A}\to\mathbb{C}$ be the evaluation map at $\idG$.
Then $e\in X$.
Moreover, the orbit of $e$ under $L$ (or $R$) consists of the evaluation morphisms at all points of $G$.
Therefore both $L^{G}e$ and $R^{G}e$ separate points in $\mathfrak{A}$, so that $\overline{L^{G}e} = \overline{R^{G}e} = X$ by the Urysohn lemma.
A version of \cite[Proposition 3.9]{MR603625} now implies that the point $e$ is generic for $\nu$ under the action $L\times R$ with respect to some left \folner{} sequence of the form $(\Theta_{N}s_{N} \times \Theta_{N})$ on $G\times G$.
Since the function $L^{g_1} R^{h_1} \xi \cdots L^{g_n} R^{h_n} \xi$ is continuous, we have
\begin{multline*}
\int_{X} L^{g_1} R^{h_1} \xi \cdots L^{g_n} R^{h_n} \xi \intd\nu\\
=
\lim_{N\to\infty} \frac1{\haar(\Theta_{N})\haar(\Theta_{N}s_{N})} \int_{(l,r)\in \Theta_{N}s_{N} \times \Theta_{N}} (\lmult_{g_1^{-1}} \rmult_{h_1^{-1}} \xi \cdots \lmult_{g_n^{-1}} \rmult_{h_n^{-1}} \xi)(L^{l}R^{r}e) \intd\haar(l) \intd\haar(r)\\
\leq
\limsup_{N\to\infty} \frac1{\haar(\Theta_{N})^{2}\modular(s_{N})} \int_{(l,r)\in \Theta_{N}s_{N} \times \Theta_{N}} (\lmult_{g_1^{-1}} \rmult_{h_1^{-1}} \ind{S} \cdots \lmult_{g_n^{-1}} \rmult_{h_n^{-1}} \ind{S})(lr^{-1}) \intd\haar(l) \intd\haar(r)\\
% =
% \limsup_{N\to\infty} \frac1{\haar(\Theta_{N})^{2}\modular(s_{N})} \int \ind{\Theta_{N}s_{N}}(l) \ind{\Theta_{N}}(r) (\ind{g_{1}^{-1}Sh_{1}} \cdots \ind{g_{n}^{-1}Sh_{n}})(lr^{-1}) \intd\haar(l) \intd\haar(r)\\
% =
% \limsup_{N\to\infty} \frac1{\haar(\Theta_{N})^{2}\modular(s_{N})} \int \modular(r) \int \ind{\Theta_{N}s_{N}}(lr) \ind{\Theta_{N}}(r) (\ind{g_{1}^{-1}Sh_{1} \cap\dots\cap g_{n}^{-1}Sh_{n}})(l) \intd\haar(l) \intd\haar(r)\\
% =
% \limsup_{N\to\infty} \frac1{\haar(\Theta_{N})^{2}\modular(s_{N})} \int_{g_{1}^{-1}Sh_{1} \cap\dots\cap g_{n}^{-1}Sh_{n}} \int \modular(r) \ind{\Theta_{N}s_{N}}(lr) \ind{\Theta_{N}}(r) \intd\haar(r) \intd\haar(l)\\
=
\limsup_{N\to\infty} \frac1{\haar(\Theta_{N})^{2}\modular(s_{N})} \int_{g_{1}^{-1}Sh_{1} \cap\dots\cap g_{n}^{-1}Sh_{n}} \ind{\Theta_{N}s_{N}} \conv \ind{\Theta_{N}}^{*} \intd\haar(l),
\end{multline*}
and we obtain the conclusion with the two-sided Reiter sequence $\Psi_{N}=\haar(\Theta_{N})^{-2}\modular(s_{N})^{-1} \ind{\Theta_{N}s_{N}} \conv \ind{\Theta_{N}}^{*}$.

Finally, when $G$ is discrete we can take $U = \{ \idG \}$, so $\xi = \ind{W} = \xi^{2}$ is an indicator function of some clopen set $B$.
\end{proof}

We remark that, in Part~\ref{thm:lcscfcp:nu} of Theorem~\ref{thm:lcscfcp}, we need to pass from the given two-sided Reiter sequence to one for which $e$ is generic in order to get an ergodic measure-preserving system.
This can be masked if one is willing to use two-sided upper Banach density.

\begin{corollary}
\label{cor:fcp-erg:lcsc}
Let $G$ be a locally compact, second countable, amenable group with a left Haar measure $\haar$.
Let $W$ be a Borel subset of $G$, let $U$ be an open neighborhood of $\idG$, and let $S\supseteq UWU$ be measurable.
Then there is a compact metric space $X$ with a continuous $G\times G$-action $L\times R$, an ergodic $L\times R$-invariant Borel probability measure $\nu$ on $X$, and a non-negative, continuous function $\xi$ on $X$ such that $\int \xi \intd\nu = \upbdens(W)$ and
\[
\upbdens(g_1^{-1} S h_1 \cap \cdots \cap g_n^{-1} S h_n) \ge \int L^{g_1} R^{h_1} \xi \cdots L^{g_n} R^{h_n} \xi \intd\nu
\]
for any $g_1,\dots,g_n,h_1,\dots,h_n$ in $G$.

If $G$ is discrete, we can take $\xi$ to be a characteristic function of a clopen subset of $X$.
\end{corollary}
\begin{proof}
Let $\Phi$ be a two-sided Reiter sequence such that $\dens_{\Phi}(W)=\upbdens(W)$.
Let $(X,\nu,L\times R)$ be the regular probability measure-preserving system and $\xi$ the continuous function given by Theorem~\ref{thm:lcscfcp}\eqref{thm:lcscfcp:nu}.
It remains to show that $\int\xi\intd\nu \leq \upbdens(W)$.
By construction we have
\begin{multline*}
\int_{X} \xi \intd\nu
=
\lim_{N\to\infty} \frac1{\haar(\Theta_{N})\haar(\Theta_{N}s_{N})} \int_{(l,r)\in \Theta_{N}s_{N} \times \Theta_{N}} \xi(lr^{-1}) \intd\haar(l) \intd\haar(r)\\
=
\lim_{N\to\infty} \frac1{\haar(\Theta_{N})\haar(\Theta_{N}s_{N})} \int_{(l,r)\in \Theta_{N}s_{N} \times \Theta_{N}} \int_{y,z} \psi(y) \ind{W}(y^{-1}lr^{-1}z) \check{\psi}(z^{-1}) \intd\haar(y)\intd\haar(z) \intd\haar(l) \intd\haar(r)\\
=
\limsup_{N\to\infty} \frac1{\haar(\Theta_{N})\haar(\Theta_{N}s_{N})} \int_{y,z} \psi(y)\psi(z) \int_{(l,r)\in \Theta_{N}s_{N} \times \Theta_{N}} \ind{W}(y^{-1}lr^{-1}z) \intd\haar(l) \intd\haar(r) \intd\haar(y)\intd\haar(z).
\end{multline*}
By the Fatou lemma this is bounded above by
\[
\int_{y,z} \psi(y)\psi(z) \limsup_{N\to\infty} \frac1{\haar(\Theta_{N})\haar(\Theta_{N}s_{N})} \int_{(l,r)\in \Theta_{N}s_{N} \times \Theta_{N}} \ind{W}(y^{-1}l(z^{-1}r)^{-1}) \intd\haar(l) \intd\haar(r) \intd\haar(y)\intd\haar(z),
\]
and the $\limsup$ above equals $\upperdens_{\Psi}(W)$ by the \folner{} property.
Since $\int\psi\intd\mu=1$, it follows that the integral is bounded above by $\upbdens(W)$.
\end{proof}

\section{Large sets in WM groups}
\label{sec:WM}
This section is dedicated to the proof of the Theorem~\ref{mainTheorem:WM:lcsc}.
We also show that, in general, Theorem~\ref{mainTheorem:WM:lcsc} fails for \emph{left} substantial sets, even for countable groups and subsets of density $1$.
We will need some facts about commuting, measure preserving actions of WM groups on probability spaces.
Recall that two actions $T_1$ and $T_2$ of $G$ are said to \define{commute} if $T_1^g T_2^h = T_2^h T_1^g$ for all $g, h$ in $G$.
Given an action $T$ of $G$ on a probability space $(X,\mathscr{B},\mu)$, denote by $T \times T$ the induced action of $G$ on the product probability space $(X \times X,\mathscr{B} \otimes \mathscr{B},\mu \otimes \mu)$.

\begin{definition}
Let $T$ be an action of a \lcsc{} group $G$ on a probability space $(X,\mathscr{B},\mu)$.
Denote by $\inv(T)$ the closed subspace of $\lp^2(X,\mathscr{B},\mu)$ consisting of functions invariant under $T$, and by $\kron(T)$ the closed subspace of $\lp^2(X,\mathscr{B},\mu)$ spanned by the finite-dimensional, $T$-invariant subspaces therein.
\end{definition}

\begin{lemma}
\label{lem:I-prod}
For any measure-preserving action $T$ of a \lcscWMa{} group $G$ on a probability space $(X,\mathscr{B},\mu)$ we have $\inv(T) = \kron(T)$ and $\inv(T \times T) = \inv(T) \otimes \inv(T)$.
\end{lemma}
\begin{proof}
The inclusion $\inv(T) \subset \kron(T)$ is immediate.
For the reverse inclusion, consider a finite-dimensional, $T$-invariant subspace $H$ of $\lp^2(X,\mathscr{B},\mu)$.
It induces a finite-dimensional representation of $G$, which must be trivial because $G$ is WM.
Therefore $H \subset \inv(T)$.

The second assertion follows from the first and the fact that $\kron(T \times T) = \kron(T) \otimes \kron(T)$, which follows from \cite{MR0174705}.
\end{proof}

\begin{definition}
Let $T$ be a measure-preserving action of a \lcsca{} group $G$ on a probability space $(X,\mathscr{B},\mu)$ and let $\mathscr{D}$ be a $T$-invariant sub-$\sigma$-algebra of $\mathscr{B}$.
We can think of $\lp^2(X,\mathscr{B},\mu)$ as an $\lp^\infty(X,\mathscr{D},\mu)$ module.
Denote by $\kron(T|\mathscr{D})$ the closed subspace of $\lp^2(X,\mathscr{B},\mu)$ spanned by the closed, finite-rank, $T$-invariant $\lp^\infty(X,\mathscr{D},\mu)$ sub-modules.
\end{definition}

\begin{lemma}
\label{lem:A-prod}
Let $T_1$ and $T_2$ be measure-preserving actions of a \lcsca{} group $G$ on probability spaces $(X_1,\mathscr{B}_1,\mu_1)$ and $(X_2,\mathscr{B}_2,\mu_2)$ respectively.
Let $\mathscr{D}_1$ and $\mathscr{D}_2$ be $T_1$ and $T_2$ invariant sub-$\sigma$-algebras of $\mathscr{B}_1$ and $\mathscr{B}_2$ respectively.
Then $\kron(T_1|\mathscr{D}_1) \otimes \kron(T_2|\mathscr{D}_2) = \kron(T_1 \times T_2|\mathscr{D}_1 \times \mathscr{D}_2)$.
\end{lemma}
\begin{proof}
The inclusion $\subset$ follows from the definition.
For the reverse inclusion, note that if $f_1$ is orthogonal to $\kron(T_1|\mathscr{D}_1)$ or $f_2$ is orthogonal to $\kron(T_2|\mathscr{D}_2)$ then $f_1 \otimes f_2$ is orthogonal to $\kron(T_1 \times T_2|\mathscr{D}_1 \otimes \mathscr{D}_2)$ by the characterization (see \cite[Theorem~4.7]{arxiv:1402.3843}, for example) of the orthogonal complements of these spaces.
\end{proof}

The next definition is \cite[Definition~2.5]{arxiv:1309.6095}, which is based on a notion introduced in \cite{MR2539560}.

\begin{definition}
Let $T_1$ and $T_2$ be commuting, measure-preserving actions of a \lcsca{} group $G$ on a probability space $(X,\mathscr{B},\mu)$.
We say that $(T_1,T_2)$ is \define{magic} if $\kron(T_1|\inv(T_2)) = \inv(T_1) \vee \inv(T_2)$.
\end{definition}

\begin{lemma}
\label{lem:product-magic}
If $G$ is a \lcscWMa{} group and $(T_1,T_2)$ is magic, then $(T_1 \times T_1,T_2 \times T_2)$ is also magic.
\end{lemma}
\begin{proof}
Applying Lemma~\ref{lem:I-prod} and then Lemma~\ref{lem:A-prod}, we have
\begin{equation*}
\kron(T_1 \times T_1|\inv(T_2 \times T_2)) = \kron(T_1|\inv(T_2)) \otimes \kron(T_1|\inv(T_2))
\end{equation*}
so the conclusion follows upon using the fact that $(T_1,T_2)$ is magic, noting that
\begin{equation*}
(\inv(T_1) \vee \inv(T_2)) \otimes (\inv(T_1) \vee \inv(T_2)) = (\inv(T_1) \otimes \inv(T_1)) \vee (\inv(T_2) \otimes \inv(T_2))
\end{equation*}
and further use of Lemma~\ref{lem:I-prod}.
\end{proof}

Given a left Reiter sequence $\Phi$ in $\lp^{1}(G,\haar)$ and a measurable map $u : g \mapsto u_g$ from $G$ to $\mathbb{C}$, write
\begin{equation*}
\clim_{g \to \Phi} u_g = \lim_{N \to \infty} \int \Phi_N(g) u_g \intd\haar(g)
\end{equation*}
if this limit exists.
Write $\uclim_{g} u_{g} = u$ if $\clim_{g \to \Phi} u_{g} = u$ for every left Reiter sequence $\Phi$,
write $\dlim_{g\to\Phi} u_{g} = u$ if $\clim_{g\to\Phi} |u_g - u| = 0$,
and write $\udlim_g u_g = u$ if $\uclim_g |u_g - u| = 0$.
Write $\inv_1$ for $\inv(T_1)$, $\inv_2$ for $\inv(T_2)$, and $\inv_{12}$ for $\inv(T_1T_2)$.

\begin{theorem}
\label{thm:wmDlim}
If $G$ is a \lcscWMa{} group, $(T_1,T_2)$ is magic, and $(g_1,g_2) \mapsto T_1^{g_1} T_2^{g_2}$ is an ergodic action of $G \times G$, then
\begin{equation*}
\begin{aligned}
&
\udlim_g \int f_0 \cdot T_1^g f_1 \cdot T_1^g T_2^g f_2\intd\mu\\
=
&
\int \condex{f_0}{\inv_1 \vee \inv_{12}} \cdot \condex{f_1}{\inv_1 \vee \inv_2} \cdot \condex{f_2}{\inv_2 \vee \inv_{12}} \intd\mu
\end{aligned}
\end{equation*}
for any $f_0,f_1,f_2$ in $\lp^\infty(X,\mathscr{B},\mu)$.
\end{theorem}
\begin{proof}
By \cite[Lemma 4.4]{arxiv:1309.6095} and our ergodicity assumptions the function
\begin{equation*}
g \mapsto \int \condex{f_0}{\inv_1 \vee \inv_{12}} \cdot T_1^g \condex{f_1}{\inv_1 \vee \inv_2} \cdot T_1^g T_2^g \condex{f_2}{\inv_2 \vee \inv_{12}} \intd\mu
\end{equation*}
is almost periodic.
Since $G$ is WM, it is in fact constant.
Thus it suffices to prove that
\begin{equation}
\label{eqn:zeroDlimForWm}
\udlim_{g} \int f_{0} \cdot T_{1}^{g} f_{1} \cdot T_{1}^g T_2^{g} f_{2} \intd\mu = 0
\end{equation}
for any $f_0,f_1,f_2$ in $\lp^\infty(X,\mathscr{B},\mu)$ satisfying one of the following conditions:
\begin{enumerate}
\item $f_{0} \perp \inv_1 \vee \inv_{12} \iff f_{0}\otimes f_{0} \perp \inv(T_1 \times T_1) \vee \kron(T_1T_2 \times T_1T_2)$
\item $f_{1} \perp \inv_1 \vee \inv_2 \iff f_{1}\otimes f_{1} \perp \inv(T_{1}\times T_{1}) \vee \inv(T_{2}\times T_{2})$
\item $f_{2} \perp \inv_2 \vee \inv_{12} \iff f_{2}\otimes f_{2} \perp \inv(T_{2}\times T_{2}) \vee \kron(T_1T_2 \times T_1 T_2)$
\end{enumerate}
where the equivalences all follow from Lemma~\ref{lem:I-prod}.
But \eqref{eqn:zeroDlimForWm} is equivalent to
\begin{equation*}
\uclim_g \int (f_0 \otimes f_0) (T_1 \times T_1)^g (f_1 \otimes f_1) (T_1T_2 \times T_1T_2)^g (f_2 \otimes f_2)\intd(\mu \otimes \mu) = 0,
\end{equation*}
which is true under any of the above conditions by Lemma~\ref{lem:product-magic} and \cite[Corollary~3.6]{arxiv:1309.6095}.
\end{proof}

\begin{corollary}
\label{cor:UDlim-erg}
Let $G$ be a \lcscWMa{} group and let $T_1,T_2$ be commuting $G$-actions on a probability space $(X,\mathscr{B},\mu)$ such that the induced $G \times G$ action $(g_1,g_2) \mapsto T_1^{g_1}T_2^{g_2}$ is ergodic.
Then
\begin{equation*}
\udlim_{g} \int f \cdot T_{1}^{g}f \cdot T_1^g T_2^{g} f \intd\mu
\geq
\left( \int f \intd\mu \right)^4
\end{equation*}
for any $0 \le f \le 1$ in $\lp^\infty(X,\mathscr{B},\mu)$.
\begin{proof}
By \cite[Corollary 4.6]{arxiv:1309.6095} the system $(X,\mathscr{B},\mu,T_{1},T_{2})$ admits an ergodic magic extension which is denoted by the same symbols.
Lifting $f$ to this extension, Theorem~\ref{thm:wmDlim} and \cite[Lemma 1.6]{MR2794947} combined yield
\begin{equation*}
\begin{aligned}
&
\udlim_g \int f \cdot T_1^g f \cdot T_1^g T_2^g f\intd\mu\\
=&
\int \condex{f}{\inv_1 \vee \inv_{12}} \condex{f}{\inv_1 \vee \inv_2} \condex{f}{\inv_2 \vee \inv_{12}} \intd\mu\\
\geq&
\int f \condex{f}{\inv_1 \vee \inv_{12}} \condex{f}{\inv_1 \vee \inv_2} \condex{f}{\inv_2 \vee \inv_{12}} \intd\mu
\geq \left( \int f \intd\mu \right)^4
\end{aligned}
\end{equation*}
as desired.
\end{proof}
\end{corollary}

We are now in a position to prove Theorem~\ref{mainTheorem:WM:lcsc}.

\begin{proof}[Proof of Theorem~\ref{mainTheorem:WM:lcsc}]
Let $E$ be a subset of $G$ that is substantial with respect to some two-sided Reiter sequence $\Phi$.
Put $g_0 = \idG$, $E_0 = E$ and $\Psi_0 = \Phi$.
Let also $K_{1}\subset K_{2}\subset \cdots$ be an exhaustion of $G$ by compact sets given by Lemma~\ref{lem:lcsc-union-of-compact}.
We construct inductively a sequence $g_i$ in $G$, a sequence $E_{i}$ of measurable subsets of $G$ with $g_{i+1} \in E_i$, and a sequence $\Psi_i$ of two-sided Reiter sequences in $G$ such that
\begin{equation}
\label{eqn:twoSidedIpInductionSet}
E_{i+1} = g_{i+1}^{-1} E_{i} \cap E_{i} \cap E_{i} g_{i+1}^{-1},
\end{equation}
the set $E_i$ is substantial with respect to $\Psi_i$ and $g_{i}\not\in K_{i}$ for all $i \ge 0$.
Assume by induction that for some $i \ge 0$, we have $g_{j}$, $E_j$ and $\Psi_j$ defined for all $0 \le j \le i$ and having the desired properties.

Since $E_i$ is substantial with respect to $\Psi_i$ we can find a symmetric, open neighborhood $U$ of the identity in $G$ and a measurable subset $W$ of $G$ with $\upperdens_{\Psi_i}(W) > 0$ such that $E_i \supset UUWUU$.
Put $S = UWU$.
Since $S$ is substantial, Part~\ref{thm:lcscfcp:nu} of Theorem~\ref{thm:lcscfcp} yields commuting $G$ actions $L$ and $R$ on a compact, metric probability space $(X,\mathscr{B},\nu)$, a two-sided Reiter sequence $\Psi_{i+1}$ in $G$, and a non-negative, continuous function $\xi$ on $X$ such that the $G \times G$ action induced by $L$ and $R$ is ergodic, $\int \xi \intd\nu \ge \upperdens_{\Psi_i}(W)$ and
\begin{equation*}
\upperdens_{\Psi_{i+1}}(S g^{-1} \cap S \cap g^{-1} S) \ge \int L^g \xi \cdot \xi \cdot R^{g^{-1}} \xi \intd\nu = \int \xi \cdot R^g \xi \cdot R^g L^g \xi \intd\nu
\end{equation*}
for all $g$ in $G$.

By Theorem~\ref{thm:wmDlim} with $T_1 = R$, $T_2 = L$ and $f = \xi$ we obtain
\begin{equation*}
\udlim_g \int \xi \cdot R^g \xi \cdot R^g L^g \xi \intd\nu \ge \upperdens_{\Psi_{i}}(W)^4
\end{equation*}
so for any $\epsilon > 0$ the set
\begin{equation*}
F_i = \{ g \in G : \upperdens_{\Psi_{i+1}} (S g^{-1} \cap S \cap g^{-1}S) \ge \upperdens_{\Psi_i}(W)^4 - \epsilon \}
\end{equation*}
has density $1$ with respect to any left Reiter sequence.
Thus $\dens_{\Psi_i}(F_i) = 1$ so in particular the intersection $F_i \cap E_i$ is not relatively compact.
Choose $g_{i+1}$ in $F_i \cap E_i \setminus K_{i+1}$ and define $E_{i+1}$ by \eqref{eqn:twoSidedIpInductionSet}.
To see that $E_{i+1}$ is substantial with respect to $\Psi_{i+1}$, note that
\begin{equation*}
E_i g_{i+1}^{-1} \cap E_i \cap g_{i+1}^{-1} E_i \supset V(S g_{i+1}^{-1} \cap S \cap g_{i+1}^{-1} S)V
\end{equation*}
where $V$ is the symmetric neighborhood $U \cap g_{i+1}U g_{i+1}^{-1} \cap g_{i+1}^{-1} U g_{i+1}$ of the identity.
This concludes the inductive construction.

It remains to prove that $\fp(g_n)$ is contained in $E$.
Note that $g_n$ belongs to $E$ for each $n$ and that \eqref{eqn:twoSidedIpInductionSet} implies
\begin{equation*}
E_i = \cap \{ \inc_\alpha(g_n)^{-1} E \dec_\beta(g_n)^{-1} \,:\, \alpha,\beta \subset \{1,\dots,i\} \textrm{ and } \alpha \cap \beta = \emptyset \}
\end{equation*}
for each $i$ in $\mathbb{N}$ by induction on $i$.
Thus $\inc_\alpha(g_n) g_{i+1} \dec_\beta(g_n)$ belongs to $E$ for any disjoint subsets $\alpha,\beta$ of $\{ 1,\dots,i\}$ as desired.
\end{proof}

We remark that Theorem~\ref{mainTheorem:WM:lcsc} immediately implies the following partition result.

\begin{corollary}
For any measurable partition $C_1 \cup \cdots \cup C_r$ of a \lcsc{}, WM, amenable group one can find $1 \le i \le r$ such that, for any open neighborhood $U$ of the identity, the set $U C_i U$ contains a two-sided finite products set.
\end{corollary}

Also, in the proof of Theorem~\ref{mainTheorem:WM:lcsc}, there are in fact many choices for each $g_i$ because $F_i \cap E_i$ has positive density with respect to a two-sided Reiter sequence.

The following example shows that Theorem~\ref{mainTheorem:WM:lcsc} does not extend to general left \folner{} sequences, even in the discrete case.

\begin{example}
\label{eg:twoSidedNecessary}
Let $A(\N)$ be the finite alternating group on $\mathbb{N}$.
As in Example~\ref{ex:wm:alt} we can view $A(\N)$ as the increasing union of the alternating groups $A_n=A(\{1,\dots,n\})$.
For $n\geq 4$ let $h_{n}:=(1,n)(2,3)\in A_{n}$ and $\Phi_{n}:=A_{n-1}h_{n}$.
This is a left \folner{} sequence, and we see from
\[
\Phi_{n} \subset \{ \sigma\in A_{n} : \sigma(1)=n \}
\]
that the sets $\Phi_{n}$ are pairwise disjoint.
Let $E:=\cup_{n\geq 5}\Phi_{n}$, so that $\dens_{\Phi}(E)=1$.
Suppose that there exists a sequence $(g_{k})$ with $\fpd(g_{k}) \subset E$.
We will show the following statement by induction on $n$.
\begin{claim}
Suppose that for some $\alpha\in\mathcal{F}$ and $n\geq 4$ we have $\dec_{\alpha}(g_{k})\in\Phi_{n}$.
Then there exists $\beta\in\mathcal{F}$ such that $\dec_{\beta}(g_{k})\in\Phi_{4}$.
\end{claim}
Since the assumption is clearly satisfied for some $n\geq 5$ and some $\alpha\in\mathcal{F}$, we obtain a contradiction.
\begin{proof}[Proof of the claim]
For $n=4$ the conclusion holds with $\beta=\alpha$.
Suppose now that the claim is known to hold up to some $n$ and assume $h:=\dec_{\alpha}(g_{k})\in\Phi_{n+1}$.
Let $i>\max\alpha$, let $j$ be such that $g_{i}\in\Phi_{j}$, and consider $D_{\alpha\cup\{i\}}(g_{k}) = g_{i} h$.
By the assumption this is an element of $E$.
Consider now the following cases.
\begin{description}
\item[$j>n+1$] We have $g_{i}h(1)=g_{i}(n+1)<j$ and $g_{i}h(j)=g_{i}(j)<j$, so that $g_{i}h\not\in E$, contradicting the assumption.
\item[$j\leq n$] The conclusion follows from the inductive assumption.
\item[$j=n+1$] In this case we have $g_{i}h(1)=g_{i}(n+1)<n+1$, so that $g_{i}h\not\in\Phi_{n+1}$.
Since $g_{i}h\in E\cap A_{n+1}$, it follows that $g_{i}h \in\Phi_{m}$ for some $m\leq n$, and the conclusion again follows from the inductive hypothesis.\qedhere
\end{description}
\end{proof}
\end{example}

In view of \cite[Theorem 2.4]{MR2561208} this implies that the set $E$ in this example has density zero with respect to any right \folner{} sequence.
As discussed in the introduction, this answers a question from \cite{arxiv:1307.0767}.

For the proof of Theorem~\ref{mainTheorem:virtuallyWM:lcsc} we need a tool that allows us to deduce that a set having positive density in a group has positive density in some coset of any cocompact subgroup.

\begin{lemma}
\label{lem:folner-seq-restriction-to-subgroup}
Let $G$ be a \lcsca{} group with a left (respectively two-sided) \folner{} sequence $\Phi$.
Let $H$ be a closed, normal, cocompact subgroup of $G$.
Then $\Phi$ has a subsequence, denoted by the same symbol, such that for almost every $z\in G/H$ and any $x\in z$ the sequence $N\mapsto x^{-1}(\Phi_{N}\cap z)$ is a left (respectively two-sided) \folner{} sequence in $H$.

If $E$ is a measurable subset of $G$ such that $\upperdens_{\Phi}(E) > 0$, then for a positive measure set of $z\in G/H$ we have $\upperdens_{\Phi\cap z}(E) > 0$.
\end{lemma}
\begin{proof}
We will prove the result for two-sided \folner{} sequences, the proof for left \folner{} sequences is nearly identical.

Consider for each $x$ in $G$ the measure $\haar_x$ defined by $\haar_x(E) = \haar_H(x^{-1}E)$ for all Borel subsets $E$ of $xH$.
This is a measure on the coset $xH$ and is invariant under left and right translation by $H$.
The measure $\haar_x$ only depends on the coset of $H$ that $x$ represents.
We can therefore define, for each $z$ in $G/H$, a measure $\haar_z$ on $G$ by $\haar_z = \haar_x$ for any $x \in z$.
After suitably normalizing the left Haar measure $\haar_{H}$, we obtain
\begin{equation*}
\int f \intd\haar = \iint f \intd\haar_{z} \intd\mu(z),
\end{equation*}
where $\mu$ denotes the Haar measure on $G/H$, for any $f \in \contc(G)$ from \cite[Theorem~2.49]{MR1397028}.
In fact, this holds for any $f$ in $\lp^1(G,\haar)$.

We have
\begin{equation*}
\haar_x(\Phi_N \symdiff g \Phi_N) \ge |\haar_x(\Phi_N) - \haar_{g^{-1}x}(\Phi_N)|
\end{equation*}
for all $N$ in $\N$ and all $g,x$ in $G$.
By the \folner{} property it follows that
\begin{equation*}
\int \frac{|\haar_{z}(\Phi_{N})-\haar_{g^{-1}z}(\Phi_{N})|}{\haar(\Phi_{N})} \intd\mu(z)
\leq \frac{\haar(\Phi_{N} \symdiff g\Phi_{N})}{\haar(\Phi_{N})}
\to 0
\end{equation*}
as $N \to \infty$.
Invariance of $\mu$ implies
\begin{equation*}
1 = \int \frac{\haar_z(\Phi_N)}{\haar(\Phi_N)} \intd\mu(z) = \int \frac{\haar_{zw}(\Phi_N)}{\haar(\Phi_N)} \intd\mu(z)
\end{equation*}
for every $w \in G/H$, so
\begin{equation*}
\int \left| 1 - \frac{\haar_{w}(\Phi_{N})}{\haar(\Phi_{N})} \right| \intd\mu(w)
\leq
\iint \frac{|\haar_{zw}(\Phi_N) - \haar_{w}(\Phi_N)|}{\haar(\Phi_N)} \intd\mu(z) \intd\mu(w) \to 0
\end{equation*}
as $N\to\infty$ by the dominated convergence theorem.
Passing to a subsequence, we may assume
\begin{equation}
\label{eq:fo-seq-meas-on-fibers}
\haar_{z}(\Phi_{N})/\haar(\Phi_{N}) \to 1
\text{ for almost every } z \in G/H.
\end{equation}

By the \folner{} condition we also have $\haar(\Phi_{N} \symdiff g\Phi_{N}h)/\haar(\Phi_{N}) \to 0$ locally uniformly for $g,h\in H$.
Let $K_{n}$ be a countable collection of relatively compact open sets that covers $H\times H$.
We have
\begin{equation*}
\int_{K_{n}} \int \frac{\haar_{z}(\Phi_{N} \symdiff g\Phi_{N}h)}{\haar(\Phi_{N})} \intd\mu(z) \intd(\haar_H \times \haar_H)(g,h) \to 0
\end{equation*}
for every $n$.
Passing to a subsequence, we may assume that the convergence holds pointwise almost everywhere for every $n$.

By Fubini's theorem it follows that for almost every $z \in G/H$ we have 
\begin{equation*}
\frac{\haar_{z}(\Phi_{N} \symdiff g\Phi_{N}h)}{\haar(\Phi_{N})} \to 0
\end{equation*}
for almost every pair $(g,h)$ in $H\times H$.
Combined with \eqref{eq:fo-seq-meas-on-fibers}, this implies
\begin{equation*}
\frac{\haar_{z}(\Phi_{N} \symdiff g\Phi_{N}h)}{\haar_{z}(\Phi_{N})} \to 0
\end{equation*}
with the same quantifiers.
By the remarks at the start of Section \ref{sec:folnerConditions}, $\Phi_{N}\cap z$ is a two-sided \folner{} sequence in the two-sided $H$ torsor $z$ for almost every $z$ in $G/H$.
It follows that $x^{-1}(\Phi_{N}\cap z)$ is a two-sided \folner{} sequence in $H$ for any $x$ in $z$.

Suppose now that $\upperdens_{\Phi}(E)>0$.
Before performing the above construction we may pass to a subsequence of $\Phi$ such that $\dens_{\Phi}(E)>0$.
This ensures $\dens_{\Phi}(E)>0$ after passing to further subsequences as required by the construction. 
Now the Fatou lemma justifies
\begin{equation*}
\begin{aligned}
0
< \upperdens_{\Phi}(E)
&=
\limsup_{N \to \infty} \frac{1}{\haar(\Phi_{N})} \int \ind{\Phi_{N}} \cdot \ind{E} \intd\haar_G\\
&=
\limsup_{N \to \infty} \int \frac{1}{\haar(\Phi_{N})} \int \ind{\Phi_{N}} \cdot \ind{E} \intd\haar_{z} \intd\mu(z)\\
&\leq
\int \limsup_{N \to \infty} \frac{1}{\haar(\Phi_{N})} \int \ind{\Phi_{N}} \cdot \ind{E} \intd\haar_{z} \intd\mu(z)\\
&=
\int \limsup_{N \to \infty} \frac{1}{\haar_{z}(\Phi_{N})} \int \ind{\Phi_{N}} \cdot \ind{E} \intd\haar_{z} \intd \mu(z) = \int \upperdens_{\Phi \cap z}(E) \intd\mu(z)
\end{aligned}
\end{equation*}
so $\upperdens_{\Phi \cap z}(E) > 0$ for a positive measure set of cosets.
\end{proof}

We now prove Theorem~\ref{mainTheorem:virtuallyWM:lcsc}.

\begin{proof}[Proof of Theorem~\ref{mainTheorem:virtuallyWM:lcsc}]
Let $H\leq G$ be a closed normal subgroup such that $H$ is WM and $G/H$ is compact.
For the second part it suffices to take a nonnull compact subset of $G/H$ that does not contain the equivalence class of $\idG$ and pull it back to $G$; the fact that this is not a set of recurrence will be witnessed by the left translation action of $G$ on $G/H$.

For the first part note that $H$ is unimodular by Lemma~\ref{lem:WMunimod}, so that $G$ is also unimodular.
Hence, for a given Borel subset $E \subset G$ of positive upper Banach density, the density is realized along some two-sided \folner{} sequence $\Phi$ by Proposition~\ref{prop:slicing}.

It remains to show that $E$ has positive upper density in one of the cosets of $H$.
By Lemma~\ref{lem:folner-seq-restriction-to-subgroup} the sequence $\Phi$ is a two-sided \folner{} sequence in almost every coset of $H$ and $\upperdens_{\Phi \cap z}(E) > 0$ for a positive measure set of cosets.
Pick a coset $z$ such that this holds and $\Phi \cap z$ is a two-sided \folner{} sequence in $z$.
Now if $UEU$ is a substantial set in $G$ and $x \in z$, then $x^{-1}UEU \cap H \supset (x^{-1}Ux \cap H) (x^{-1}E \cap H) (U\cap H)$ is a substantial set in $H$, and we can apply Theorem~\ref{mainTheorem:WM:lcsc} to it.
\end{proof}

\begin{proof}[Proof of Theorem~\ref{thm:hindman-G0-cocompact}]
As in the proof of Theorem~\ref{mainTheorem:virtuallyWM:lcsc}, shifting $S$ (on the left or on the right) we may assume $S=U^{9}EU^{2}$, where $E=UE_{0}$ has positive upper density with respect to some left \folner{} sequence $\Phi$ in $G$, $E_{0}\subset G_{0}$ has positive upper density in $G_{0}$ with respect to the \folner{} sequence $N\mapsto\Phi_{N}\cap G_{0}$, and $U$ is a symmetric neighborhood of the identity in $G$.

Put $g_0 = \idG$, $S_0 = S$, $U_{0}=U$, and $\Psi_0 = \Phi$.
Let also $K_{1}\subset K_{2}\subset \cdots$ be an exhaustion of $G$ by compact sets given by Lemma~\ref{lem:lcsc-union-of-compact}.
We construct inductively
\begin{enumerate}
\item a nested sequence of subsequences $\Psi_{i}$ of $\Phi$
\item a decreasing sequence $S_{i}$ of substantial subsets of $G$ such that $S_{i} \supseteq U_{i}^{10}E_{i}U_{i}^{2}$ with $\upperdens_{\Psi_{i}\cap G_{0}}(E_{i}) > 0$ and symmetric relatively compact neighborhoods of identity $U_{i}$ in $G$, and
\item a sequence $g_i$ in $G$ with $g_{i+1} \in S_i\setminus K_{i}$ such that
\begin{equation}
\label{eqn:oneSidedIpInductionSet}
S_{i+1} = g_{i+1}^{-1} S_{i} \cap S_{i}.
\end{equation}
\end{enumerate}
This suffices to conclude $\fpi(g_n) \subset S$.

Assume by induction that for some $i \ge 0$, we have $g_{j}$, $S_{j}$, $U_{j}$, and $E_j$ defined for all $0 \le j \le i$ and having the desired properties.
Applying the one-sided correspondence principle \cite[Theorem 1.1]{MR2561208} to $U_{i}(U_{i}E_{i})$, we obtain a measure-preserving action $T$ of $G$ on a separable probability space $(X,\mathscr{B},\mu)$ and a positive function $\xi$ in $\lp^\infty(X,\mathscr{B},\mu)$ such that
\begin{equation*}
\upperdens_\Phi(U_{i}^{2}E_{i} \cap g^{-1} U_{i}^{2}E_{i}) \ge \int \xi \cdot T^g \xi \intd\mu
\end{equation*}
for all $g$ in $G$.
Let $\epsilon = 10^{-1} | \int \xi \intd\mu |^{2}$.

Recall that
\[
\lp^{2}(X,\mathscr{B},\mu) = \kron(T) \oplus \wm(T),
\]
where $\wm(T)$ is the closed subspace consisting of the functions $f$ such that for every $\phi\in \lp^{\infty}(X)$ we have
\begin{equation}
\label{eq:wm-subspace}
\uclim_{g} \big| \int \phi \cdot T^{g}f \intd\mu \big| = 0.
\end{equation}
Applying \eqref{eq:wm-subspace} with $\phi=\xi$ and $f=\xi-\condex{\xi}{\kron(T)}$ we see that
\begin{equation}
\label{eqn:kronOrthoSmall}
\left| \int \xi \cdot T^{g}(\xi-\condex{\xi}{\kron(T)}) \intd\mu \right| < \epsilon
\end{equation}
for a set of $g\in G$ with density $1$.
Let $W\subset U_{i}$ be a neighborhood of identity such that
\begin{equation}
\label{eqn:kronShiftSmall}
\| \condex{\xi}{\kron(T)} - T^{g}\condex{\xi}{\kron(T)} \| < \epsilon
\end{equation}
for every $g\in W$.
Since almost periodic functions on $G$ are trivial on $G_0$, the above inequality holds for all $g\in WE_{i}$.
The latter set has positive upper density with respect to $\Phi$ in $G$, so that there is some $g\in WE_{i}\setminus K_{i}\overline{U_{i}}^{2}$ for which both inequalities \eqref{eqn:kronOrthoSmall} and \eqref{eqn:kronShiftSmall} hold.

For this $g$ the set $U_{i}^{2}E_{i} \cap g^{-1}U_{i}^{2}E_{i}$ has positive upper density in $G$.
Since the latter set is contained in $U_{i}^{2}G_{0}$, by Lemma~\ref{lem:folner-seq-restriction-to-subgroup} there exists $h\in U^{2}$ such that $h^{-1}(U_{i}^{2}E_{i} \cap g^{-1}U_{i}^{2}E_{i}) \cap G_{0}$ has positive upper density in $G_{0}$.
It follows that
\[
E_{i+1} := U^{4}_{i}E_{i} \cap g_{i+1}^{-1} U_{i}^{2}E_{i} \cap G_{0}
\]
has positive upper density in $G_{0}$ with $g_{i+1} = gh \in S_{i}\setminus K_{i}$.
Therefore $S_{i} \cap g_{i+1}^{-1} S_{i}$ is a substantial set of the same special form as $S_{i}$, and we can continue the induction.
\end{proof}

\section{Large sets in groups with von Neumann kernel not cocompact}
\label{sec:notvirtuallyWM}

In this section we extend Straus's example to \lcsca{} groups whose von Neumann kernel is not cocompact by proving Theorem~\ref{thm:straus-set}.
The construction involves the following approximate version of the monotone convergence theorem for the finitely additive measure defined by a \folner{} sequence.

%Choosing $K$ to be a compact neighborhood of the identity in Theorem~\ref{thm:straus-set}, we obtain a substantial set that does not contain a set of measurable recurrence.
%Recall from Section~\ref{sec:Preliminaries} that the hypothesis of Theorem~\ref{thm:straus-set} is satisfied for discrete $G$ that are not virtually WM.

\begin{lemma}
\label{lem:union-cofinite}
Let $G$ be a locally compact, second countable, amenable group and let $\Phi$ be a left, right, or two-sided \folner{} sequence in $G$.
Let $i \mapsto A_i$ be a sequence of Borel subset of $G$ such that $\dens_{\Phi}(A_{i})$ exists for every $i$, the sequence $i \mapsto \dens_{\Phi}(A_{i})$ is summable, and $\dens_{\Phi}(A_1 \cup \cdots \cup A_n)$ exists for every $n$.
Then there exist cocompact subsets $A'_{i}\subset A_{i}$ such that $\dens_{\Phi}(C) = \lim \dens_{\Phi}(A_1 \cup \cdots \cup A_n)$, where $C=\cup \{ A_i' : i \in \mathbb{N} \}$.
\end{lemma}
\begin{proof}
If $G$ is compact, then $\dens_{\Phi}=\haar$, and the conclusion holds with $A'_{i}=A_{i}$ by the monotone convergence theorem.
Hence we may assume that $G$ is not compact.

For each $i \in \N$ there is some $s_i$ in $\N$ such that
\begin{equation*}
\left| \frac{\haar( A_i \cap \Phi_N)}{\haar(\Phi_N)} - \dens_\Phi(A_i) \right| < \frac{1}{2^i}
\end{equation*}
whenever $N \ge s_i$.
We may assume that $s_{i+1} > s_i$ for all $i \in \N$.
We will show that the conclusion holds for $A'_i = A_i \setminus (\Phi_1 \cup \cdots \cup \Phi_{s_i})$.
Note that $\frac{\haar( A'_i \cap \Phi_N)}{\haar(\Phi_N)} < \dens_\Phi(A_i) + \frac{1}{2^i}$ for every $N$.

Define $B_n = A_1 \cup \cdots \cup A_n$ and $C_n = A'_1 \cup \cdots \cup A'_n$.
Let $\alpha = \lim \dens_\Phi(B_n)$.
Since $C_n \subset B_n$ and $B_n \setminus C_n$ is contained in a compact set we have $\dens_\Phi(C_n) = \dens_\Phi(B_n)$.
From $C_n \subset C_{n+1}$ it follows that $\lowerdens_\Phi(C) \ge \alpha$.
Thus it remains to show that $\upperdens_\Phi(C) \le \alpha$.
For every $J$ and every $N$ we have
\begin{align*}
\frac{\haar(C \cap \Phi_N)}{\haar(\Phi_N)}
& \le \frac{\haar(C_J \cap \Phi_N)}{\haar(\Phi_N)} + \sum_{i = J+1}^\infty \frac{\haar(A'_i \cap \Phi_N)}{(\Phi_N)}\\
& \le \frac{\haar(B_J \cap \Phi_N)}{\haar(\Phi_N)} + \sum_{i=J+1}^\infty \dens_\Phi(A_i) + \frac{1}{2^i}.
\end{align*}
Letting $N\to\infty$ we obtain
\[
\upperdens_{\Phi}(C)
\leq \dens_{\Phi}(B_{J}) + 2^{-J} + \sum_{i=J+1}^\infty \dens_\Phi(A_i).
\]
Now, letting $J\to\infty$, we obtain $\upperdens_{\Phi}(C) \leq \alpha$ as required.
\end{proof}

\begin{proof}[Proof of Theorem~\ref{thm:straus-set}]
Fix a left \folner{} sequence $\Phi$ in $G$ and $\epsilon > 0$.
Let $K_{1}\subset K_{2} \subset \cdots$ be an increasing sequence of compact sets whose union is $G$, as in Lemma~\ref{lem:lcsc-union-of-compact}.
For every $n$ let $f_{n}$ be the almost periodic function on $G$ given by Lemma~\ref{lem:slowly-varying-ap} with the compact set $K_{n}$ and $\epsilon/2^{n+10}$.
Let $X_{n} = \overline{ \{ \lmult_g f_{n} \,:\, g \in G \} }$, then $(X_{n},G)$ is a topologically transitive, equicontinuous topological dynamical system.
Therefore $(X_{n},G)$ is minimal and uniquely ergodic by Theorem~\ref{uniqueErgodicitySegal}.
Let $\mu_{n}$ be the unique invariant Borel probability measure on $(X_{n},G)$.
By minimality $\mu_n$ has full support.
Let $\epsilon/2^{n+5} < r_{n} < \epsilon/2^{n+4}$ be such that $\mu_{n}(\partial B_{r_{n}}(f_{n})) = 0$.
(Such $r_{n}$ exist because the boundaries are pairwise disjoint.)
Since the action $l$ is isometric on $X_{n}$ and the range of $f_{n}$ is $\epsilon/2^{n+10}$-dense in $[0,1]$, there are at least $2^{n+1}/\epsilon$ disjoint images of $U_{n}:=B_{r_{n}}(f_{n})$ in $X_{n}$, so $\mu_{n}(U_{n}) \leq \epsilon/2^{n+1}$.
Put $A_n = \ret_{U_n}(f_{n})$, so that $\dens_{\Phi}(A_{n}) = \mu_{n}(U_{n}) \leq \epsilon/2^{n+1}$ by Lemma~\ref{densityIsMeasure}.

In order to apply Lemma~\ref{lem:union-cofinite} we need to prove that
\begin{equation*}
B_n = A_1 \cup \cdots \cup A_n
\end{equation*}
has density for every $n$.
To do this consider the action $L$ of $G$ on $X_{1}\times\dots\times X_{n}$ induced by applying $l$ in each coordinate.
Let $Z$ be the orbit closure of $(f_{1},\dots, f_{n})$ under this action.
By Theorem~\ref{uniqueErgodicitySegal} the topological dynamical system $(Z,G)$ is minimal and uniquely ergodic.
We have
\begin{align*}
B_n
&= \cup_{i=1}^n \{ x \in G \,:\, \lmult_{x} f_{i} \in U_i \}\\
&= \{ x \in G \,:\, L_x(f_{1},\dots, f_{n}) \in \pi_1^{-1} U_1 \cup \cdots \cup \pi_n^{-1} U_n \}
\end{align*}
so $B_n$ is a set of return times in a uniquely ergodic dynamical system.
Let $\nu$ be the unique invariant probability measure on $(Z,G)$.
Each of the coordinate projections $\pi_i : Z \to X_{i}$ intertwines the actions $L$ and $l$.
Since $X_{i}$ is uniquely ergodic, this implies that $\pi_i \nu = \mu_{i}$ for each $i$.
Thus
\begin{equation*}
\nu( \partial (\pi_1^{-1} U_1 \cup \cdots \cup \pi_n^{-1} U_n)) \le \mu_{1}(\partial U_1) + \cdots + \mu_{n}(\partial U_n) = 0
\end{equation*}
and therefore $\dens_\Phi(B_n)$ exists by Lemma~\ref{densityIsMeasure}.

It follows from Lemma~\ref{lem:union-cofinite} that there exist cocompact subsets $A_{n}' \subset A_{n}$ whose union $C$ has density at most $\epsilon$.
We claim that the set $E:=G\setminus C$ satisfies the conclusion of the theorem.
Indeed, let $K$ be an arbitrary symmetric compact subset of $G$.
It suffices to show that the complement of $KEK$ is an $\SetRec^{*}$ set.

By construction we have $K\subset K_{n}$ for some $n$.
Put $D_n = G \backslash A_n$ and $D_n' = G \backslash A_n'$.
We have $E \subset D_n'$ so it suffices to prove that $G \setminus (K D_n' K)$ is $\SetRec^{*}$.
%We have $(KEK)^{\complement} \supset (K(A'_{n})^{\complement}K)^{\complement}$, so it suffices to show that the set $(K(A'_{n})^{\complement}K)^{\complement}$ is $\SetRec^{*}$.
By Corollary~\ref{cor:cofin-pres-srs} this is equivalent to $G \setminus (K D_{n} K)$ being an $\SetRec^{*}$ set.
Now
\[
G \setminus (K D_{n} K)
=
\bigcap_{g,h\in K}
g A_{n} h
=
\bigcap_{g,h\in K}
\ret_{\lmult_{g} B_{r_{n}}(f_{n})}(\lmult_{h^{-1}} f_{n})
\supseteq
\ret_{B_{r_{n}/2}(f_{n})}(f_{n})
\]
since $f_{n}$ is $\epsilon/2^{n+10}$-invariant under $K$ and $G$ acts isometrically on $X_{n}$.
The latter set is $\SetRec^{*}$ by Lemma~\ref{lem:top-ret-is-sec-ret}.
\end{proof}

In Theorem~\ref{thm:straus-set} the set $E$ depends on the \folner{} sequence.
It is natural to ask whether it is possible for $E$ to satisfy the conclusion of Theorem~\ref{thm:straus-set} and have positive density with respect to all left \folner{} sequences.
To see that this is impossible we will show that any set satisfying the conclusion of Theorem~\ref{thm:straus-set} is not piecewise syndetic.
In view of Example~\ref{eg:fpiMeasRec} it suffices to show the following.

\begin{lemma}
\label{psContainsShiftedFP}
Let $G$ be a \lcsc{} group that is not compact.
Then for every piecewise-syndetic subset $P$ of $G$ there exists $k\in G$ such that for every neighborhood $U$ of $\idG$ the set $UkP$ contains a set of the form $\fpi(g_n)$ with $g_n \to \infty$.
\begin{proof}
Let $P$ be a piecewise-syndetic subset of $G$.
Fix a compact subset $K$ of $G$ such that $T=KP$ is thick.
Let $n \mapsto K_n$ be a sequence of increasing, compact subsets of $G$ that cover $G$ (see Lemma~\ref{lem:lcsc-union-of-compact}).

Using Lemma~\ref{lem:largeShiftsForThick} with $K = \{ \idG \}$, choose $g_1 \in G \setminus K_1$ such that $g_1 \in T$.
Assume by induction that we have found $g_1,\dots,g_n$ in $G$ such that $H = \{ \inc_\alpha(g_n) \,:\, \emptyset \ne \alpha \subset \{1,\dots,n\} \}$ is a subset of $T$ and $g_i \notin K_i$ for each $1 \le i \le n$.
By Lemma~\ref{lem:largeShiftsForThick} there is some $g_{n+1}$ in $G \setminus (H \cup \{\idG \})^{-1}K_{n+1}$ such that $(H \cup \{\idG \})g_{n+1} \subset T$.
It follows that $T$ contains $\fpi(g_n)$ and $\inc_\alpha(g_n) \to \infty$ as $\alpha\to\infty$.

For each finite set $\emptyset\neq\alpha\subset\N$ let $k_{\alpha}\in K$ be such that $k_{\alpha}^{-1}\inc_\alpha(g_n) \in P$.
Since $G$ is metrizable, by \cite[Theorem 1.3]{MR833409} there is a sub-IP-ring such that $\lim_{\alpha}k_{\alpha} = k$ exists.
In particular, for every symmetric neighborhood $U$ of $\idG$ we have $k_{\alpha}\in Uk$ for sufficiently large $\alpha$.
Hence $\inc_\alpha(g_n) \in k_{\alpha}P \subset UkP$ as required.
\end{proof}
\end{lemma}

Call any subset $E$ of $G$ satisfying the conclusion of Theorem~\ref{thm:straus-set} a \define{Straus set}.
It follows from Lemma~\ref{psContainsShiftedFP} that any Straus set $E$ is not piecewise syndetic.
In particular, $E$ is not syndetic, and therefore its complement is thick.
(This follows from the proof of Lemma~\ref{lem:synd-thick-dual}.)
Now, any thick set $T$ has full density with respect to some \folner{} sequence in $G$.
Indeed, for any \folner{} sequence $\Phi$ in $G$, one can find for each $N$ some $h_N \in G$ such that $\Phi_N h_N \subset T$, so that $T$ has full density with respect to the \folner{} sequence $N \mapsto \Phi_N h_N$.
Thus it is impossible for a Straus set to have positive upper density with respect to every left \folner{} sequence.

\section{Non piecewise-syndetic sets with large density}
\label{sec:Extras}

By the remarks at the end of Section~\ref{sec:notvirtuallyWM} we see that Straus sets are not piecewise-syndetic.
However, Straus sets only exist in amenable groups whose von Neumann kernel is not cocompact.
In this section we prove Theorem~\ref{thm:large-lower-non-PWS}, which states that it is possible in any \lcsca{} group that is not compact to construct a non piecewise-syndetic set with positive lower density.
The proof of Theorem~\ref{thm:large-lower-non-PWS} reqires an ample supply of syndetic sets.

\begin{lemma}
\label{lem:greedy}
Let $G$ be a group, $S\subset G$ be any set, and $F\subset G$ be a non-empty set.
Then there exists a subset $S'\subset S$ such that
\begin{enumerate}
\item\label{item:cover} $F^{-1} F S' \supseteq S$;
\item\label{item:disjoint-shifts} $f_{1}S' \cap f_{2}S' = \emptyset$ for any $f_{1},f_{2}\in F$, $f_{1}\neq f_{2}$.
\end{enumerate}
\begin{proof}
The collection
\begin{equation*}
\{ T \subset S \,:\, f_1 T \cap f_2 T = \emptyset \textrm{ for any } f_1 \ne f_2 \in F \}
\end{equation*}
is closed under increasing unions, hence it contains a maximal element $S'$ by Zorn's Lemma.
Suppose that condition \eqref{item:cover} fails for $S'$, that is, there exists $s\in S\setminus F^{-1}F S'$.
We will show that the set $S'':=S'\cup\{s\}$ also satisfies condition \eqref{item:disjoint-shifts}, thereby contradicting the maximality of $S'$.
To this end let $f_{1},f_{2}$ be two distinct elements of $F$.
We have
\[
f_{1}S'' \cap f_{2}S''
=
(f_{1}S' \cap f_{2}S') \cup (f_{1}S' \cap f_{2}\{s\}) \cup (f_{1}\{s\}\cap f_{2}S') \cup (f_{1}\{s\} \cap f_{2}\{s\})
=
\emptyset
\]
by the assumptions that $S'$ satisfies \eqref{item:disjoint-shifts} and that $s\not\in F^{-1}FS$, as required.
\end{proof}
\end{lemma}

\begin{lemma}
\label{lem:small-syndetic-lcsc}
Let $G$ be a \lcsca{} group that is not compact.
Then for every relatively compact neighborhood $U$ of the identity there exists a decreasing sequence $S_{1}\supset S_{2}\supset \cdots$ of discrete, syndetic subsets of $G$ such that $\upbdens(US_{n})\to 0$ as $n\to\infty$. 
\end{lemma}
\begin{proof}
Put $S_{0}=G$ and let $g_{1},g_{2},\dots$ be a sequence in $G$ such that the sets $g_{n}U$ are pairwise disjoint.
Suppose that $S_{n}$ has been constructed for some $n$.
Let $S_{n+1}$ be given by Lemma~\ref{lem:greedy} with $F=g_{1}U \cup\dots\cup g_{n+1}U$ and $S=S_{n}$.
The set $F$ is relatively compact, so it follows from syndeticity of $S_{n+1}$ and Part~\ref{item:cover} of Lemma~\ref{lem:greedy} that $S_{n+1}$ is syndetic.
By Part~\ref{item:disjoint-shifts} of Lemma~\ref{lem:greedy} the sets $g_{i}Us$ are pairwise disjoint for $i=1,\dots,n+1$ and $s\in S_{n+1}$.
This implies that $S_{n+1}$ is discrete and $\upbdens(US_{n+1})\leq\frac{1}{n+1}$.
\end{proof}

The final ingredient in the proof of Theorem~\ref{thm:large-lower-non-PWS} is an appropriate version of Lemma~\ref{lem:union-cofinite}.

\begin{lemma}
\label{lem:cofinite2}
Let $G$ be a \lcsca{} group with a (left, right, or two-sided) \folner{} sequence $\Phi$ and let $A_{i}$ be a sequence of measurable subsets of $G$.
Then for every $\epsilon>0$ there exist cocompact subsets $A_{i}'\subset A_{i}$ such that $\upperdens_{\Phi}(\cup_{i}A_{i}') \leq \sum_{i}\upperdens_{\Phi}(A_{i}) + \epsilon$.
\end{lemma}
\begin{proof}
By definition of the upper density, for every $i$ there exists $s_{i}$ such that
\[
\frac{\haar(A_{i} \cap \Phi_{N})}{\haar(\Phi_{N})} < \upperdens_{\Phi}(A_{i}) + \frac{\epsilon}{2^{i}}
\]
for every $N>s_{i}$.
Consider the cocompact subsets $A_{i}' := A_{i} \setminus \cup_{n=1}^{s_{i}}\Phi_{n}$.
Then for every $N$ we have
\begin{multline*}
\haar(\bigcup_{i} A_{i}' \cap \Phi_{N})
= \haar(\bigcup_{i : s_{i}<N} A_{i}' \cap \Phi_{N})
\leq \haar(\bigcup_{i : s_{i}<N} (A_{i} \cap \Phi_{N}))
\leq \sum_{i : s_{i}<N} \haar( A_{i} \cap \Phi_{N} )\\
< \sum_{i : s_{i}<N} (\upperdens_{\Phi}(A_{i}) + \epsilon/2^{i}) \haar(\Phi_{N})
\leq \haar(\Phi_{N}) \big( \sum_{i}\upperdens_{\Phi}(A_{i}) + \epsilon \big)
\end{multline*}
as required.
\end{proof}

\begin{proof}[Proof of Theorem~\ref{thm:large-lower-non-PWS}]
Let $(g_i)_{i \in \N}$ be a dense subset of $G$, let $U$ be a symmetric, relatively compact neighborhood of $\idG$, and let $G \supset S_{1} \supset S_{2} \supset \cdots$ be the decreasing sequence of syndetic sets from Lemma~\ref{lem:small-syndetic-lcsc}.
Passing to a subsequence we may assume $\sum_{i} \upbdens(S_{i}) < \epsilon$.
By Lemma~\ref{lem:cofinite2} we obtain cocompact subsets $S_{i}' \subset US_{i}$ such that $\upperdens_{\Phi}(\cup \{ g_{i} S_{i}' : i \in \N \}) < \epsilon$.
Consider the set
\begin{equation*}
Q := G \setminus \cup \{ g_{i} S_{i}' \,:\, i \in \N \}.
\end{equation*}
It follows from the construction that $Q$ has lower density at least $1-\epsilon$ with repsect to $\Phi$.
We claim that $Q$ is not piecewise-syndetic.
Suppose $Q$ is piecewise-syndetic, that is, that there exists a compact set $K\subset G$ such that $T:= KQ$ is a thick set.
By compactness we have $K\subset Ug_1^{-1} \cup \cdots \cup Ug_N^{-1}$ for some $N\in\N$.
Thus
\begin{equation*}
KQ
\subseteq
U g_1^{-1} Q \cup \cdots \cup U g_N^{-1} Q
\subseteq
\cup_{i=1}^N \cup_{u \in U} u g_i^{-1} (G \setminus g_i S_i')
\end{equation*}
and so
\begin{equation*}
G \setminus KQ
\supseteq
%\cap_{i=1}^{N} \cap_{u\in U} (ug_{i}^{-1}(G\setminus g_{i}S_{i}'))^{\complement}\\
%=
\cap_{i=1}^{N} \cap_{u\in U} uS_{i}'
\supseteq
S_{N} \setminus \cup_{i=1}^{N} U(US_{i}\setminus S_{i}').
\end{equation*}
This is a cocompact subset of the syndetic set $S_{N}$, hence a syndetic set, contradicting thickness of $KQ$.
Hence $Q$ is not piecewise syndetic.
\end{proof}

We next show, using the (rather deep) Jewett--Krieger theorem for countable, amenable groups, that if $G$ is a countable, infinite, amenable group then the set $E$ obtained in Theorem~\ref{thm:large-lower-non-PWS} can be taken to have density.

\begin{theorem}
\label{thm:large-non-PWS}
For any countably infinite, amenable group $G$, any left \folner{} sequence $\Phi$ in $G$, and any $\epsilon>0$ there is a subset $Q$ of $G$ with $\dens_\Phi(Q) > 1 - \epsilon$ that is not piecewise syndetic.
\end{theorem}
\begin{proof}
Fix $\epsilon > 0$.
Consider an ergodic action of $G$ on a non-atomic space, for example the action of $G$ on $\{ 0,1 \}^G$ given by $(g \cdot \ind{B})(x) = \ind{B}(xg)$ equipped with a Bernoulli measure.
By Rosenthal's Jewett--Krieger theorem \cite{Rosenthal-thesis} this action admits a topological model $(X,G)$ that is uniquely ergodic.
Let $\mu$ be the unique invariant probability measure.
Since $X$ is infinite, the measure $\mu$ is non-atomic.
Fix $x$ in the support of $\mu$.
For each $n$ in $\N$, Lemma~\ref{lem:smallZeroBoundaryBalls} yields an open ball $A_n \subset X$ centered at $x$ with $\mu(A_n)<\epsilon/2^n$ and $\mu(\partial A_n)=0$.
Since $x$ is in the support of $\mu$, each $A_n$ has positive measure.
Applying Lemma~\ref{densityIsMeasure}, we obtain $\dens_\Psi(\ret_{A_n}(x)) = \mu(A_n)$ for every left \folner{} sequence $\Phi$.
Thus, for each $n$, the set $S_n := \ret_{A_n}(x)$ has positive density with respect to every left \folner{} sequence.
It follows from this that each of the sets $S_n$ is syndetic.

Let $(g_i)_{i \in \N}$ be an enumeration of $G$.
Put $B_i = g_1 S_1 \cup \cdots g_i S_i$.
We have
\begin{equation*}
B_i = \{ g \in G \,:\, gx \in g_1 A_1 \cup \cdots \cup g_i A_i \} = \ret_{Y_i}(x)
\end{equation*}
where $Y_i := g_1 A_1 \cup \cdots \cup g_i A_i$.
Since $\mu(\partial Y_i) = 0$ it follows from Lemma~\ref{densityIsMeasure} that $\dens_\Phi(B_i)$ exists.
Applying Lemma~\ref{lem:union-cofinite}, we obtain cofinite subsets $S_i' \subset S_i$ such that $\dens_\Phi(\cup \{ g_i S_i' \,:\, i \in \N \}) < \epsilon$.

Put $Q = G \setminus \cup \{ g_i S_i' \,:\, i \in \N \}$.
The density of $Q$ exists and is at least $1 - \epsilon$.
Arguing as in the proof of Theorem~\ref{thm:large-lower-non-PWS} shows that $Q$ is not piecewise syndetic.
To this end, note that $S_{i_1} \cap \cdots \cap S_{i_k} = S_l$ where $l = \max \{ i_1,\dots,i_k \}$.
Thus the intersection is syndetic.
Hence its cofinite subset $S_{i_1}' \cap \cdots \cap S_{i_k}'$ is also syndetic.
\end{proof}

We will use sets whose existence is ensured by Theorem~\ref{thm:large-non-PWS} in order to construct non-trivial actions of $G$.
Denote by $\ind{\emptyset}$ the function $G \to \{ 0,1 \}$ mapping every element of $G$ to $0$.

\begin{proposition}
\label{prop:non-atomic-measure}
Let $G$ be a countably infinite amenable group and let $\Phi$ be a \folner{} sequence in $G$. Suppose $Q$ is a subset of $G$ that has positive upper density but is not piecewise-syndetic.
Let $X$ be the orbit closure of $\ind{Q}$ in $\{ 0,1 \}^G$ under the action of $G$ on $\{ 0,1 \}^G$ given by $(g \cdot \ind{B})(x) = \ind{B}(xg)$.
Then $\{ \ind{\emptyset} \}$ is the only minimal subsystem of $X$, but there is a non-atomic $G$-invariant probability measure on $X$.
\end{proposition}
\begin{proof}
Let $\ind{B}$ be any point in $X$ and assume that $B$ is syndetic.
Let $F$ be a finite set such that $FB = G$.
Let $H$ be any finite subset of $G$.
Since $\ind{B}$ is in the orbit closure of $\ind{Q}$ we can find $g \in G$ such that $\ind{Qg}$ and $\ind{B}$ agree on the finite set $F^{-1}H$.
In particular,
\begin{align*}
FQg
&\supset F(F^{-1}H \cap B)
\supset \bigcup_{f\in F} f (f^{-1}H \cap B)
= H \cap FB
= H,
\end{align*}
so $FQ$ contains $Hg^{-1}$.
Since $H$ was arbitrary, $Q$ must be piecewise-syndetic, a contradiction.
Hence no point of $X$ can correspond to a syndetic set.

Suppose now that $Y$ is a minimal subsystem of $X$ different from $\{ \ind{\emptyset} \}$.
Then the set $C = \{ \omega \in \{ 0,1 \}^G \,:\, \omega(\idG) = 1 \}$ has non-empty intersection with $Y$.
Since $Y$ is minimal, every point $\ind{B}\in Y$ visits $C$ syndetically.
In particular, every point corresponds to a syndetic set, a contradiction.

For the second assertion, let $\Psi$ be a sub-sequence of $\Phi$ such that $\dens_\Psi(Q) = \lowerdens_\Phi(Q)$ and let $\mu$ be any limit point of the sequence
\begin{equation*}
\mu_N = \frac{1}{|\Psi_N|} \sum_{g \in \Psi_N} \delta_{g \cdot \ind{Q}}
\end{equation*}
of probability measures on $X$.
We have $\mu(C) = \dens_\Psi(Q) > 0$, so $\mu(\{ \ind{\emptyset} \})<1$.

Hence there is a $G$ invariant probability measure with no point mass at $0$.
Suppose now that it has a point mass at some other point $\ind{B}$.
Then the orbit of $\ind{B}$ is finite, so it is a minimal subsystem of $X$ different from $\{ \ind{\emptyset} \}$, a contradiction.
\end{proof}

We conclude with an example demonstrating that, in general, the invariant probability measure in Proposition~\ref{prop:non-atomic-measure} will not have full support on $X$, even in the case of a $\Z$-action.

\begin{example}
Let $Q \subset \Z$ be a non piecewise-syndetic set of positive lower density.
(One can take, for instance, a Straus set in $\Z$.)
Since $Q$ is not syndetic, its characteristic function $\ind{Q}$ has two consecutive zeroes.
Translating $Q$ if necessary we may assume that $0,1\not\in Q$.
Consider now the set $Q':=2Q \cup (2Q+1)\cup \{1\}$.
This set still has positive lower density, and is not piecewise-syndetic since it is contained in the union of three translates of the non piecewise-syndetic set $2Q$.
Moreover, $1$ is the only member of $Q'$ both of whose neighbors are not in $Q'$.
Let now $X\subset \{0,1\}^{\Z}$ be the orbit closure of $\ind{Q'}$ and let $C := \{ \omega \in \{ 0,1 \}^\Z \,:\, \omega(0)=\omega(2)=0,\omega(1) = 1 \}$.
Then the cylinder sets $X\cap (C+n) = \{\ind{Q'+n}\}$ are disjoint singletons that are mapped to each other under the action of $\Z$.
Hence any $\Z$-invariant probability measure assigns zero measure to them.
Since they are open, an invariant probability measure cannot have full support.
\end{example}

\printbibliography
\end{document}